\documentclass[reqno,a4paper]{amsart}
\usepackage{eucal,amsfonts,amssymb,amsmath,amsthm,epsfig,mathrsfs}
  \usepackage{paralist}
   \usepackage{txfonts}
    \usepackage{overpic}
     \usepackage{cases}
      \usepackage{subfig}

\usepackage{amscd,amsxtra}
\usepackage{enumerate}
\usepackage{latexsym}

\allowdisplaybreaks

 \makeatletter \@addtoreset{equation}{section}

\makeatother \makeatletter

\title[Two-dimensional stability analysis in a HIV model]
      {Two-dimensional stability analysis in a HIV model with quadratic logistic growth term}

\author{Claude-Michel Brauner}
\address{School of Mathematical Sciences, Xiamen University, 361005 Xiamen, China
and Institut de Math\'ematiques de Bordeaux, Universit\'e de Bordeaux, 33405 Talence cedex, France}
\email{cmbrauner@gmail.com, claude-michel.brauner@u-bordeaux1.fr}

\author{Xinyue Fan}
\address{College of Science, Guizhou University, 550025 Guiyang, China. Former affiliation: School of Mathematical Sciences, Xiamen University, 361005 Xiamen, China}
\email{fan.xinyue@163.com}

\author{L. Lorenzi}
\address{Dipartimento di Matematica e Informatica, Universit\`a degli Studi di Parma, Parco Area delle Scienze 53/A, I-43124 Parma, Italy.}
\email{luca.lorenzi@unipr.it}


\thanks{This work was partially supported by a grant of the Fujian Administration of Foreign Expert Affairs, China.}

\newtheorem{theorem}{Theorem}[section]

\newtheorem{remark}[theorem]{Remark}{\rm}
\newtheorem{proposition}[theorem]{Proposition}

\newtheorem{corollary}[theorem]{Corollary}

\newcommand{\e}{\varepsilon}

\newcommand{\Om}{\Omega}

\newcommand\re{{\rm Re}}

\newcommand{\mt}{\mu_T}
\newcommand{\mi}{\mu_I}
\newcommand{\mv}{\mu_V}

\newcommand{\Rmnum}[1]{\expandafter\@slowromancap\romannumeral #1@}

\newcommand\C{{\mathbb C}}
\newcommand\R{{\mathbb R}}

\newcommand\N{{\mathbb N}}
\renewcommand{\phi}{\varphi}

\begin{document}

\begin{abstract}
We consider a Human Immunodeficiency Virus (HIV) model with a logistic growth term and continue
the analysis of the previous article \cite{FB10}. We now take the viral diffusion in a two-dimensional environment.
The model consists of two ODEs for the concentrations of the target T cells, the infected cells, and a parabolic PDE
for the virus particles.
We study the stability of the uninfected and infected equilibria, the occurrence of Hopf bifurcation
and the stability of the periodic solutions.
\end{abstract}

\maketitle

\allowdisplaybreaks

\section{Introduction}\label{intro}
Over the past thirty years, there has been much research in the mathematical modeling of Human
Immunodeficiency Virus (HIV), the virus which causes AIDS (Acquired Immune Deficiency Syndrome).
The research directions have been twofold: (i) the epidemiology of AIDS; (ii) the immunology of HIV
as a pathogen. We are interested in the latter approach.

The major target of HIV infection is a class of lymphocytes, or white blood cells, known as CD4$^+$
T cells. When the CD4$^+$ T-cell count, which is normally around 1000 mm$^{-3}$, reaches 200 mm$^{-3}$
or below in an HIV-infected patient, then that person is classified as having AIDS.

Mathematical models have been proved valuable in understanding the in vivo dynamics of the virus.
A gamut of models have been developed to describe the immune system, its interaction with HIV,
and the decline in CD4$^+$ T cells. They have contributed significantly to the understanding of HIV basic biology.

Recently, the effect of spatial diffusion has been taken in account in HIV modeling.
Funk et al. \cite{funk2005spatial} introduced a discrete model: they adopted a
two-dimensional square grid with $21 \times 21$ sites and assumed
that the virus can move to the eight nearest neighboring sites.
K. Wang et al. \cite{wang2007propagation} generalized Funk's model.
They assumed that the hepatocytes can not move under normal
conditions and neglected their mobility, whereas virions can move
freely and their motion follows a Fickian diffusion.
In \cite{brauner2011heterogeneous}, two of the authors considered a two-dimensional heterogenous environment:
the basic reproductive ratio is generalized
as an eigenvalue of some Sturm-Liouville problem.
Furthermore, in the case of an alternating structure of viral sources,
the classical approach via ODE systems is justified via a homogenized limiting environment.

In this article, we consider  a HIV model which takes the viral diffusion into account in a homogeneous
two-dimensional environment, and includes a quadratic logistic growth term as previously proposed in
\cite{PKDB93,PN99} to consider the homeostatic process for the CD4$^+$ T-cell count. The model reads:
\begin{align}
\frac{\partial T}{\partial t}&=\alpha - \mt T+rT\left(1-\frac{T}{T_{\max}}\right) - \gamma V T,\label{pde1} \\
\frac{\partial I}{\partial t}&= \gamma V T - \mi I,\label{pde2} \\
\frac{\partial V}{\partial t}&=N \mi I- \mv V+ d_V \Delta V.\label{pde3}
\end{align}%
The spatial domain is denoted by $\Omega_{\ell}=(0,\ell)\times(0,\ell)$,
periodic boundary conditions are prescribed for $V$ . Since the system \eqref{pde1}-\eqref{pde3}
defines a dynamical system or semiflow, we will also use
the abstract notation ${\mathbf X}(t)= (T(t),I(t),V(t))$.

Our aim is to  continue the analysis of the previous paper \cite{FB10}, where
we studied the system \eqref{pde1}-\eqref{pde3} when $d_V=0$. We refer to \cite{FB10}
for an extended introduction to the
biological issues. In brief, we recall that $T$ and $I$ denote the respective
concentrations of uninfected and infected CD4$^+$ T cells. The
concentration of free virus particles, or virions, is $V$ (for the
sake of simplicity, we call $V$ the virus).
In \eqref{pde1}, $r$ is the
average specific T-cell growth rate obtained in the absence of
population limitation. The term
$1-T/T_{\max}$ shuts off T-cell growth as the population level
$T_{\max}$ is approached from below. Here $\mt$ is the natural death rate of CD4$^+$
T  cells, the term $\gamma V T$ models the rate at which free virus
infects a CD4$^+$ T cell. The infected cells die at a rate $\mi$
and produce free virus during their life-time at a rate $N$. In
addition, $\mv$ is the death rate of the virus.
According to the literature
(see e.g.,  Table \ref{table1} or \cite{ciupe2006estimating} where $\mt=0.01,~\mi=0.39$),
we assume the following biologically relevant
hypothesis:
\begin{equation}
\label{bio-hypoth}
\mi >\mt.
\end{equation}
Note that the quantity $r-\mt$, the net T-cell proliferation rate, needs not to be positive
(see \cite[p. 86]{PKDB93}).
\begin{table}[!h]
\tabcolsep 0pt \caption{Parameters and Variables} \vspace*{-12pt}
\begin{center}
\def\temptablewidth{1\textwidth}
{\rule{\temptablewidth}{1pt}}
\begin{tabular*}{\temptablewidth}{@{\extracolsep{\fill}}ccccccc}
& Parameters~$\&$~Variables  & Values \\   \hline
 Dependent& &  \\
 variables  & &  \\
 $T$  &Uninfected ~$CD4^+$ ~T-cell ~population~& $mm^{-3}$   \\
 $I$  &Infected ~$CD4^+$ ~T-cell ~density~& $mm^{-3}$ ~~  \\
 $V$  &HIV ~population ~size~& $mm^{-3}$    \\
Parameters~$\&$ & & \\
Constants &  & \\
     $r$  &Proliferation~rate~of~the~$CD4^+$~T-cell~population& 0.2 $day^{-1}$~~ \\
     $N$  &Number~of~virus~produced~by~infected~cells & 1000   \\
$\alpha$  &Production~rate~for~uninfected~$CD4^+$~T~cells &1.5 $day^{-1}mm^{-3}$   \\
$\gamma$  &Infection~rate~of~uninfected~$CD4^+$~T~cells &0.001 $day^{-1}mm^{3}$   \\
$T_{\max}$ &Maximal~population~level~of~$CD4^+$~~T~~cells~at & 1500 $mm^{-3}$   \\
 &which~the~$CD4^+$~T-cell~proliferation ~~shuts~~off&\\
 $\mt$  &Death~rate~of~uninfected~$CD4^+$~T-cell~population & 0.1 $day^{-1}$   \\
 $\mi$  &Death~rate~of~infected~$CD4^+$~T-cell~population & 0.5 $day^{-1}$   \\
 $\mv$  &Clearance~rate~of~free~virus & 10 $day^{-1}$\\
 Derived & &  \\
 variable & &  \\
 $T_0$   &$CD4^+$ ~T-cell ~population ~for ~HIV~negative ~persons~& $mm^{-3}$ \\
 \end{tabular*}\\
 {\rule{\temptablewidth}{1pt}}
       \end{center}
       \label{table1}
       \end{table}

From a mathematical viewpoint:\\
\noindent (i) $N>0$ and $r \ge 0$ are parameters;

\noindent (ii) the quantities (with associated dimension) $\alpha$, $\gamma$, $\mi$, $\mt$
and $\mv$ are fixed positive numbers throughout the paper;

\noindent (iii) $T_{\max}$ is a large perturbation parameter, larger than any finite combination
of  $\alpha$, $\gamma$, $\mi$, $\mt$ and $\mv$ of the same dimension ($mm^{-3}$).
In particular, this hypothesis contains the condition $T_{\max} >\alpha/\mt$ of \cite[p. 85]{PKDB93}.

The paper is organized as follows: in Section \ref{equilibria}, we prove that System \eqref{pde1}-\eqref{pde3}
admits, for any value of the parameters $r$ and $N$, the uninfected steady state ${\bf X}_u = (T_u,0,0)$
and that, in a region of the space of parameters, there exists also another steady-state solution,
the so-called infected steady state ${\bf X}_i = (T_i,I_i,V_i)$, where $T_i$, $I_i$ and $V_i$ are positive.
In the parameter space, we define the regions ${\mathcal U}$ and ${\mathcal I}$ (this latter being the region where the
infected steady-state exists), respectively for uninfected (the reproductive ratio
is such that $R_0 <1$) and infected ($R_0 >1$). We recall that  the basic reproductive ratio
denotes the average number of infected T cells derived from one infected T cell
(\cite{diekmann2000mathematical}). We prove that the uninfected steady state is asymptotically stable in ${\mathcal U}$,
and unstable in ${\mathcal I}$.

In \cite{FB10}, we have exhibited an unbounded subdomain $\mathcal P$ in $\mathcal I$ in which the positive
infected equilibrium becomes unstable whereas it is asymptotically stable in the rest of $\mathcal I$.
In this unstable region, the levels of the various cell types and virus particles oscillate, rather than
converging to steady values. This subdomain $\mathcal P$ may be biologically interpreted as a perturbation
of the infection by a specific or unspecific immune response against HIV. In Section
\ref{stability infected}, we consider the linearization around ${\mathbf X}_i=(T_i,I_i,V_i)$
of System \eqref{pde1}-\eqref{pde3} with Jacobian matrix ${\mathscr L}_i$. A modal expansion of the resolvent
equation enables us to construct a {\it finite number} of subdomains ${\mathcal P}_k$ ($k=0,\ldots,K_2$) in $\mathcal I$,
that form a monotone non-increasing sequence (for the inclusion) with ${\mathcal P}_0 = {\mathcal P}$.
It turns out that the infected equilibrium ${\mathbf X}_i= (T_i,I_i,V_i)$ is asymptotically stable for
$(N,r) \in {\mathcal I}\setminus {\mathcal P}$ and unstable in the interior of ${\mathcal P}$.
Therefore the stability issue is governed by the $0$-th mode, hence similar to the case without
viral diffusion ($d_V=0$). As a matter of fact, we are unable to confirm Funk et al. \cite{funk2005spatial},
who suggested that the presence of a spatial structure enhances population stability with respect
to non-spatial models (see also \cite{brauner2011heterogeneous}).

In Section \ref{hopf-instability}, we take the logistic parameter $r$ as bifurcation parameter
and prove the existence of Hopf bifurcations at the boundary $\partial {\mathcal P}$. Since the
system is only partially dissipative, the resolvent operator associated to the realization $L_i$
of ${\mathscr L}_i$ is not compact and therefore the proof demands more attention: it relies on
the analyticity of the semigroup $\exp(t{L}_i)$ (see e.g., \cite{hassard1981theory}, \cite{marsden1976hopf}).
Next, we perform a nonlinear analysis at the Hopf points via the Center Manifold theorem.
It turns out that
the bifurcating periodic solutions are independent of the space variables.

Numerical illustrations are presented in Section \ref{numerics}. Finally, for the sake of completeness,
we recall in an Appendix some basic facts about the eigenvalues of the two-dimensional Laplace operator
with periodic boundary conditions and some Sturm-Liouville operators.
\medskip

\paragraph{\bf Notation}
For any $\ell>0$ we denote by $L^2$ the
usual space of square-integrable functions $f:(0,\ell)^2\to\mathbb R$.
The square $(0,\ell)^2$ will be simply denoted by
$\Omega_{\ell}$.
By $H^k_{\sharp}$ ($k=1,2,\ldots$) we denote the closure in $H^k$
(the subset of $L^2$ of all the functions whose
distributional derivatives up to $k$-th order are in $L^2$)
of the space
$C^k_{\sharp}$ of all $k$-th continuously differentiable functions
$f:\R^2\to\R$ which are periodic with period $\ell$ in each
variable. The space $H^k_{\sharp}$ is endowed with
its Euclidean norm.
If $X$ is any of the previous spaces, we write $X_{\mathbb C}$ to denote the space of
complex-valued functions $f$ such that ${\rm Re}\,f$ and ${\rm Im}\,f$ are in $X$.
The norm in $X_{\mathbb C}$ is defined in the natural way:
$\|f\|_{X_{\mathbb C}}^2=\|{\rm Re}\,f\|_{X}^2+\|{\rm Im}\,f\|_{X}^2$.
If ${\bf v}$ is a vector of $\C^3$, we denote  by $v_1$, $v_2$ and $v_3$ its components. Similarly,
if ${\bf f}$ is a function defined in $\Omega_{\ell}$ with values in $\R^3$ (resp. $\C^3$),
we denote by $f_1$, $f_2$ and $f_3$ its components. If the components
of the vector ${\bf v}$ are complex numbers, we denote by $\overline{\bf v}$ the vector whose
components are the conjugates of the components of ${\bf v}$.
The Euclidean inner product in $L^2_{\mathcal C}\times L^2_{\mathcal C}\times L^2_{\mathcal C}$ is denoted
by $(\cdot,\cdot)_2$, i.e.,
\begin{eqnarray*}
({\bf f},{\bf g})_2=\int_{\Omega_{\ell}}(f_1\overline{g_1}+f_2\overline{g_2}+f_3\overline{g_3})dxdy,
\end{eqnarray*}
for any ${\bf f}, {\bf g}\in L^2_{\mathcal C}\times L^2_{\mathcal C}\times L^2_{\mathcal C}$.
Finally, we denote by $Id$ the identity operator, and by $(\cdot)^+$ the
positive part of the number in brackets.

\section{Equilibria}\label{equilibria}
\setcounter{equation}{0}
In this section we are devoted to determine the non-negative equilibria of System
\eqref{pde1}-\eqref{pde3}, i.e., the solutions $(T,I,V)\in
L^2\times L^2\times H^2_{\sharp}$ to the system
\begin{align}
&\alpha - \mt T+rT\left(1-\frac{T}{T_{\max}}\right) - \gamma V T=0,\label{pde1-1} \\
&\gamma V T - \mi I=0,
\label{pde2-1} \\
&N \mi I- \mv V+ d_V \Delta V=0.
\label{pde3-1}
\end{align}%

To state the first main result of this section, let us introduce some functions and a few notation.

By ${\bf X}_u$ and ${\bf X}_i$ we denote, respectively, the function whose entries $T_u$, $I_u$ and $V_u$ are given by
\begin{eqnarray*}
T_{u}=T_0,
\qquad I_{u} = 0, \qquad\;\, V_{u} = 0,
\end{eqnarray*}
where
\begin{eqnarray*}
T_0=T_0(r)=\frac{r-\mt+\sqrt{(r-\mt)^2+\frac{4\alpha r}{T_{\max}}}}{2r}T_{\max},
\end{eqnarray*}
and the function whose entries are given by $T_i$, $I_i$ and $V_i$, where
\begin{align*}
T_{i} =& \frac{\mv}{\gamma N},\\
I_{i}=&\frac{\alpha}{\mi}-\frac{\mt\mv}{\gamma
\mi N}+\frac{\mv r}{\gamma \mi N}\left (1-\frac{\mv}{\gamma NT_{\max}}\right ),\\
V_{i} =&\frac{\alpha N}{\mv}
-\frac{\mt}{\gamma}+\frac{r}{\gamma}\left
(1-\frac{\mv}{\gamma NT_{\max}}\right ).
\end{align*}

We further introduce two sets which will play a fundamental role in all our analysis, namely the
uninfected and infected regions ${\mathcal U}$ and ${\mathcal I}$
in the parameter space, which are defined by
\begin{eqnarray*}
{\mathcal U} =\{(N,r): N>0,\; r\geq 0,\; R_0(N,r)<1\},
 \qquad{\mathcal I} = \{(N,r): N>0,\; r\geq 0,\; R_0(N,r)>1\},
\end{eqnarray*}
where
\begin{eqnarray*}
R_0(N,r) = \frac{\gamma N T_0(r)}{\mv}
\end{eqnarray*}
is the reproduction ratio.

The interface $R_0(N,r)=1$ between the two regions ${\mathcal U}$ and ${\mathcal I}$ is the graph
of the mapping
\begin{equation}
N_{\rm crit}(r)=\frac{\mv
\left(\mt-r+\sqrt{(\mt-r)^2+\frac{4\alpha
r}{T_{\max}}}\right)}{2\alpha \gamma}=\frac{\mv}{\gamma T_0(r)},
\label{N-crit}
\end{equation}
which is decreasing by virtue of the condition $T_{\max}>\alpha\mt^{-1}$. Its image is the interval
$\left (\frac{\mv}{\gamma T_{\max}},\frac{\mt\mv}{\alpha\gamma}\right ]$.
Inverting the roles of $N$ and $r$, it is useful to define the inverse mapping:
\begin{eqnarray*}
r_{\rm crit}(N) = \frac{ (\mt\mv-\alpha\gamma  N)^{+}} {{\mv}
\left( 1-{\frac {\mv}{\gamma NT_{\max}}} \right)},\qquad  N>N_{\rm crit}(+\infty)=\frac{\mv}{\gamma T_{\max}}.
 \end{eqnarray*}

As it has been stressed in the introduction, throughout the paper we assume that
\begin{equation}
\label{bio-hypoth-1}
\mi >\mt.
\end{equation}

To prove the following theorem we assume also that
\begin{eqnarray}\label{hyp-Tmax}
T_{\max} \geq T_{\max}^0(\alpha,\gamma,\mi,\mt,\mv),
\end{eqnarray}
where $T_{\max}^0$ is fixed and large, and depends only on the quantities in brackets.


\begin{theorem}
\label{thm-2.1}
The following properties are satisfied:
\begin{enumerate}[\rm (i)]
\item
in $\mathcal U$ the uninfected steady state ${\mathbf X}_u=(T_u,I_u,V_u)$
is the only non-negative equilibrium;
\item
in $\mathcal I$ there
exist two non-negative equilibria, respectively ${\mathbf X}_u$ and the infected steady state
${\mathbf X}_i=(T_i,I_i,V_i)$;
\item
it holds $0<T_u<T_{\max}$ and, in ${\mathcal I}$, $0<T_i<T_{\max}$;
\item
${\bf X}_u$ and ${\bf X}_i$ are all the possible equilibria to System \eqref{pde1}-\eqref{pde3}, with all the components being
non-negative.
\end{enumerate}
\end{theorem}

\begin{proof}
(i) Suppose that ${\bf X}=(T,I,V)$ is a solution to System \eqref{pde1-1}-\eqref{pde3-1}. Then, from \eqref{pde1-1} we deduce that
\begin{equation}
T=\frac{r-\mt-\gamma V+\sqrt{(r-\mt-\gamma V)^2+\frac{4\alpha r}{T_{\max}}}}{2r}T_{\max}.
\label{T-V}
\end{equation}
Replacing the expression of $I$ given by \eqref{pde2-1} in \eqref{pde3-1} and, then, using \eqref{T-V},
we obtain the following self-contained nonlinear equation for $V$:
\begin{equation}
d_V\Delta V-\mv V+\Phi(V)=0,
\label{self-cont}
\end{equation}
where
\begin{eqnarray*}
\Phi(V)=\frac{\gamma NV}{2r}\left (r-\mt-\gamma V+\sqrt{(r-\mt-\gamma V)^2+\frac{4\alpha r}{T_{\max}}}\right )T_{\max}.
\end{eqnarray*}
As it is easily seen any solution to \eqref{self-cont} in $H^2_{\sharp}$ leads to a solution to System \eqref{pde1}-\eqref{pde3}.
Moreover, from any non-negative solution to Equation \eqref{self-cont} we can obtain an
equilibrium to System \eqref{pde1}-\eqref{pde3} will all the components non-negative in $\Omega_\ell$.
Hence, we can limit ourselves to looking for non-negative solutions $V\in H^2_{\sharp}$ to Equation \eqref{self-cont}.

Clearly, \eqref{self-cont} admits the trivial function $V\equiv 0$ as a solution. This solution leads to the equilibrium
${\bf X}_u$.

(ii) Let us look for other positive constant solutions to Equation \eqref{self-cont}.
We are thus lead to look for solutions to the equation
$\Phi(V)-\mv V=0$ which are non-negative.

A straightforward computation reveals that $V_i$ is
the unique solution to such an equation. Moreover, for any fixed $r>0$, $V_i$ is positive if and only if
$N>N_{\rm crit}(r)$ (see \eqref{N-crit}) i.e., if and only if $(r,N)\in {\mathcal I}$.
In this case, replacing $V=V_i$ into \eqref{T-V}
and \eqref{pde2-1}, we immediately conclude that the function ${\bf X}_i$ is an equilibrium of System
\eqref{pde1}-\eqref{pde3}.

(iii) Showing that $T_u<T_{\max}$ is just an exercise. On the hand, the inequality $T_i<T_{\max}$ in ${\mathcal I}$
follows from the definition of $T_i$ observing that, if $(r,N)\in {\mathcal I}$, then
\begin{eqnarray*}
N>N_{\rm crit}=\frac{\mv}{\gamma T_u}>\frac{\mv}{\gamma T_{\max}}.
\end{eqnarray*}

(iv) To prove that ${\bf X}_u$ and ${\bf X}_i$ are the only equilibria of System \eqref{pde1}-\eqref{pde3} with all the components
being non-negative,
we adapt to our situation a method due to H.B. Keller \cite{keller}.
We argue  by contradiction. We suppose that ${\bf X}=(T,I,V)$ is a solution to System \eqref{pde1-1}-\eqref{pde3-1}
with $V$ non-negative in $\Omega_{\ell}$ and not identically vanishing, and such that $V\neq V_i$.
Let us set $W:=V-V_i$. Since both $V$ and $V_i$ are solutions to \eqref{self-cont}, clearly $W\in H^2_{\sharp}$
and solves the equation
\begin{equation}
d_V \Delta W - \mv W+\Lambda(V,V_i)W=0,
\label{eq-w}
\end{equation}
where
\begin{align*}
\Lambda(x,y)
=&
\frac{\gamma N}{2r}\left (r-\mt-\gamma x+\sqrt{(r-\mt-\gamma x)^2+\frac{4\alpha r}{T_{\max}}}\right )T_{\max}
-\frac{\gamma^2NT_{\max}}{2r}y\\
&+\frac{\gamma^2NT_{\max}}{2r}y\int_0^1\frac{\gamma(tx+(1-t)y)-r+\mt}{\sqrt{[r-\mt-\gamma (tx+(1-t)y)]^2+\frac{4\alpha r}{T_{\max}}}}dt,
\end{align*}
for any $x,y\ge 0$.

Let $\lambda_{\max}(V,V_i)$ and $\lambda_{\max}(V,0)$ denote
the maximum eigenvalues in $L^2$ of the operators
$d_V\Delta+\Lambda(V,V_i)Id$ and $d_V \Delta
+\Lambda(V,0)Id$, respectively.
By Corollary \ref{cor-app}, we know that
\begin{eqnarray*}
\lambda_{\max}(V,0)=-\inf_{\psi \in H^1_{\sharp}, \psi \not\equiv
0} \left\{\frac{d_V\int_{\Omega_{\ell}}|\nabla\psi|^2 dxdy -
\int_{\Omega_{\ell}}\Lambda(V,0)\psi^2 dxdy}
{\int_{\Omega_{\ell}}\psi^2 dxdy} \right\},
\end{eqnarray*}
and
\begin{eqnarray*}
\lambda_{\max}(V,V_i)=-\inf_{\psi \in H^1_{\sharp}, \psi \not\equiv
0} \left\{\frac{d_V\int_{\Omega_{\ell}}|\nabla\psi|^2 dxdy -
\int_{\Omega_{\ell}}\Lambda(V,V_i)\psi^2 dxdy}
{\int_{\Omega_{\ell}}\psi^2dxdy} \right\}.
\end{eqnarray*}
We now observe that
\begin{align*}
\Lambda(V,V_i)-\Lambda(V,0)=&
\frac{\gamma^2NT_{\max}}{2r}V_i\left (-1+
\int_0^1\frac{\gamma(tV+(1-t)V_i)-r+\mt}{\sqrt{[r-\mt-\gamma (tV+(1-t)V_i)]^2+\frac{4\alpha r}{T_{\max}}}}dt\right )\\
\le &
\frac{\gamma^2NT_{\max}}{2r}V_i\left (-1+
\int_0^1\frac{|\gamma(tV+(1-t)V_i)-r+\mt|}{\sqrt{[r-\mt-\gamma (tV+(1-t)V_i)]^2+\frac{4\alpha r}{T_{\max}}}}dt\right )\\
\le &-\frac{\gamma^2NT_{\max}}{2r}V_i\left (1
-\frac{\gamma(\|V\|_{\infty}+V_i)+r+\mt}{\sqrt{[r+\mt+\gamma (\|V\|_{\infty}+V_i)]^2+\frac{4\alpha r}{T_{\max}}}}\right )
 =:-C,
\end{align*}
since the function $x\mapsto x(x^2+4\alpha r/T_{\max})^{-1/2}$ is increasing in $[0,+\infty)$. Here, we have taken advantage of the Sobolev embedding theorem to infer that $V\in H^2_{\sharp}$ is continuous in $\overline{\Omega_{\ell}}$.

Note that the constant $C$ is positive. From this remark we can easily infer that
\begin{eqnarray*}
\lambda_{\max}(V,V_i)\le\lambda_{\max}(V,0)-C.
\end{eqnarray*}

Since $W$ satisfies \eqref{eq-w} and it does not identically vanish in
$\overline{\Omega_{\ell}}$, $\mv\le\lambda_{\max}(V,V_i)<\lambda_{\max}(V,0)$.

To get to a contradiction, we now rewrite the equation satisfied by $V$ in the following way:
\begin{eqnarray*}
d_V\Delta V+\Lambda(V,0)V-\lambda_{\max}(V,0)V=(\mv-\lambda_{\max}(V,0))V:=Z.
\end{eqnarray*}
Fredholm alternative implies that $Z$ should be orthogonal to the function
$\psi$ which spans the
eigenspace associated with the eigenvalue $\lambda_{\min}(V,0)$ of the operator
$d_V\Delta+\Lambda(V,0)Id$. But this can not be the case. Indeed,
by Corollary \ref{cor-app} the function $\psi$ does not change sign in $\overline{\Omega_{\ell}}$.
Moreover, since $V$ is non-negative in $\Omega_{\ell}$ and it does not identically vanish in $\Omega_{\ell}$
and, in addition, $\mv<\lambda_{\max}(V,0)$,
$Z$ is non-positive
and it does not identically vanish in $\Omega_{\ell}$. Hence, $\psi$ is not orthogonal to $Z$.
\end{proof}
\begin{figure}[htp]
\begin{center}
\begin{overpic}[width=9cm,height=6.5cm]{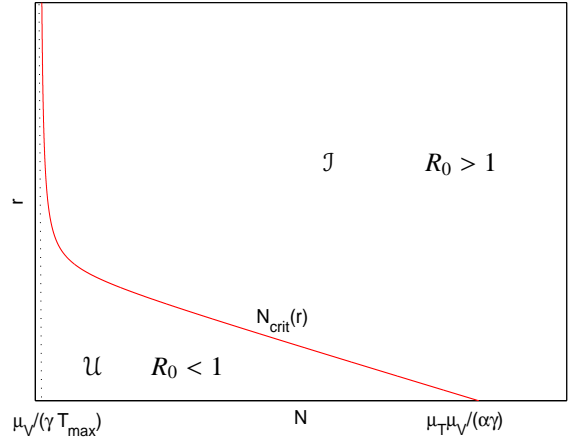}
\put(20,12){$\mathcal{U}$}
\put(55,42){$\mathcal{I}$}
\put(70,42){$R_0>1$}
\put(30,12){$R_0<1$}
\end{overpic}
\end{center}
\vskip -.4truecm
\caption{Profile of the curve $r\mapsto N_{crit}(r)$ (i.e. $R_0=1$) which defines
the two domains $\mathcal U$ and $\mathcal I$. With the values of
Table \ref{table1}, $N_{crit}$ decreases from $\mt\mv/\alpha\gamma
=666.67$ to $\mv/\gamma T_{\max}= 6.67$.}
 \label{curve Ncrit}
\end{figure}

\section{Stability of the equilibria}
\label{stability infected}

In this section we are going to study the stability of the equilibria ${\bf X}_u$ and ${\bf X}_i$.
We begin by studying the stability of the uninfected equilibrium ${\bf X}_u$.

\begin{theorem}
\label{stab-uninf}
The following properties are satisfied:
\begin{enumerate}[\rm (i)]
\item
in the domain ${\mathcal U}$ the
uninfected equilibrium ${\mathbf X}_u$ is asymptotically stable;
\item
in the domain ${\mathcal I}$ the uninfected equilibrium ${\mathbf X}_u$ is unstable.
\end{enumerate}
\end{theorem}
\begin{proof}
To avoid cumbersome notation, throughout the proof
we do not stress explicitly the dependence of the functions and operators on $r$ and $N$.

We prove the statement showing that the linearized stability principle
(see e.g., \cite[Chpt. 5, Cor. 5.1.6]{henry1981geometric}) applies to our situation.
For this purpose, we begin by observing that, for any $(N,r)$ the linearization around
${\mathbf X}_u$ of Problem \eqref{pde1}-\eqref{pde3} is
associated with the linear operator $\boldsymbol{{\mathcal L}}_u$ defined by
\begin{eqnarray*}
\boldsymbol{{\mathcal L}}_u=
   \begin{pmatrix}
    \left (r-\mt-\frac{2rT_u}{T_{\max}}\right )Id& 0 & -\gamma T_u Id \\
     0 & -\mi Id& \gamma T_u Id\\[2mm]
     0 &  \mi N Id & d_V\Delta-\mv Id \\
  \end{pmatrix}.
\end{eqnarray*}
Its realization $\boldsymbol{L}_u$ in $(L^2_{\mathbb C})^3$ with domain $D(\boldsymbol{L}_u)=L^2_{\mathbb C}\times
L^2_{\mathbb C}\times H^2_{\sharp,\mathbb C}$ generates an analytic strongly continuous
semigroup. Indeed, $\boldsymbol{L}_u$ is a bounded perturbation of the diagonal
operator
\begin{eqnarray*}
\boldsymbol{{\mathcal A}}=
\begin{pmatrix}
\left (r-\mt-\frac{2rT_u}{T_{\max}}\right )Id & 0 & 0 \\[2mm]
0 &- \mi Id & 0 \\[2mm]
0 & 0 & d_V\Delta-\mv Id
\end{pmatrix},
\end{eqnarray*}
defined in $L^2_{\mathbb C}\times L^2_{\mathbb C}\times H^2_{\sharp,\mathbb C}$, which is clearly
sectorial since all its entries are. Hence, we can apply \cite[Prop.
2.4.1(i)]{lunardi1995analytic} and conclude that $\boldsymbol{L}_u$ is sectorial. Since
$H^2_{\sharp,\mathbb C}$ is dense in $L^2_{\mathbb C}$, the associated analytic semigroup
is strongly continuous.

To complete the proof, we need to study the spectrum of the operator $\boldsymbol{L}_u$. We fix
$\lambda\in\C$ and consider the resolvent equation $\lambda {\bf X}-\boldsymbol{L}_u{\bf X}={\bf F}$,
where ${\bf X}=(T,I,V)\in L^2_{\mathbb C}\times L^2_{\mathbb C}\times H^2_{\sharp,\mathbb C}$ and
${\bf F}=(F_1,F_2,F_3)$ is a given function in $L^2_{\mathbb C}\times L^2_{\mathbb C}\times L^2_{\mathbb C}$.
Writing the previous equation componentwise gives
\begin{equation}
\left\{
\begin{array}{l}
\left (r-\mt-\frac{2r T_{u}}{T_{\max}}\right )T-\gamma T_uV=\lambda T-F_1,\\[3mm]
-\mi I+\gamma T_uV=\lambda I-F_2,\\[3mm]
\mi N I+d_V\Delta V-\mv V=\lambda V-F_3.
\end{array}
\right.
\label{sist-equ}
\end{equation}
If $\lambda\neq -\mi$ we can use the second equation to write $I$ in terms of $V$. Substituting it in
the last equation we
get the following self-contained equation for $V$:
\begin{equation}
d_V\Delta V-\left (\lambda+\mv-\frac{\gamma\mi N T_u}{\lambda+\mi}\right )V=-F_3-\frac{\mi N}{\lambda+\mi}F_2.
\label{eq-autov}
\end{equation}

We recall that the spectrum of the realization $A$ of the Laplace operator in $L^2_{\mathbb C}$,
with $H^2_{\sharp,\mathbb C}$ as a domain consists of eigenvalues only and it is given by
$$
\sigma(A)=\left\{-\frac{4 \pi^2}{\ell^2}(k_1^2+k_2^2):k_1,k_2\in \mathbb{N}\right\},
$$
(see Appendix \ref{Laplace eigenvalues}). Hence, if $c(\lambda):=\lambda+\mv-\frac{\gamma\mi N T_u}{\lambda+\mi}$
does not belong to $\sigma(A)$, then Equation
\eqref{eq-autov} admits a unique solution $V\in H^2_{\sharp,\mathbb C}$.
A straightforward computation shows that $c(\lambda)$ is real if and only if $\lambda$ is real.
Moreover, the function $\lambda\mapsto c(\lambda)$ is strictly increasing in $[0,+\infty)$ and
$c(0)=\mv-\gamma N T_u$ is positive if $(N,r)\in {\mathcal U}$. Hence, if $(N,r)\in {\mathcal U}$, then
$c(\lambda)\notin\sigma(A)$ for any $\lambda$ with non-negative real part, and
Equation \eqref{eq-autov} is uniquely solvable.

We can now uniquely determine $I\in L^2_{\mathbb C}$ from the second equation in \eqref{sist-equ}.
Finally, from the first equation in \eqref{sist-equ}, observing that
\begin{eqnarray*}
r-\mt-\frac{2rT_u}{T_{\max}}=-\sqrt{(r-\mt)^2+\frac{4\alpha r}{T_{\max}}}<0,
\end{eqnarray*}
we can uniquely determine $T\in L^2_{\mathbb C}$. We have so proved that any $\lambda\in\mathbb C$ with non-negative real part is
in the resolvent set of the operator $\boldsymbol{L}_u$, if $(N,r)\in {\mathcal U}$.
In view of the linearized stability principle this implies that the trivial uninfected
solution to System \eqref{pde1}-\eqref{pde3} is asymptotically stable.

Let us now suppose that $(N,r)\in {\mathcal I}$ and prove that $\boldsymbol{L}_u$ admits an eigenvalue with positive real part.
As above, we are led to consider the function $\lambda\mapsto c(\lambda)$. Now, $c(0)<0$. Hence, there exists $\lambda_*>0$ such that $c(\lambda_*)=0\in\sigma(A)$.
This implies that Problem \eqref{sist-equ}, with $F_1=F_2=F_3=0$ and $\lambda=\lambda_*$, admits a non trivial solution, i.e., $\lambda_*\in\sigma(\boldsymbol{L}_u)$.
Again, the linearized stability principle implies that $(T_u,0,0)$ is unstable.
This completes the proof.
\end{proof}

The issue of the stability of the infected solution is obviously much more complicated.
For notational convenience we sort the eigenvalues of the realization $A$ of the Laplacian
in $L^2_{\mathbb C}$, with $H^2_{\sharp,\mathbb C}$ as a domain (see Appendix \ref{Laplace eigenvalues}),
into a non-increasing sequence $\{-\lambda_k\}$. Similarly, we denote by $\tilde e_k$ the eigenfunction associated with
the eigenvalue $\lambda_k$. This allows ut to expand any function $f\in L^2_{\sharp,\C}$ into the Fourier series
$$
f=\sum_{k=0}^{+\infty}f_k\tilde e_k,
$$
where $f_k$ denotes the $k$-th Fourier coefficient (with respect to the system $(\tilde e_k)$) of $f$.
As it is observed in the proof of
Theorem \ref{laplace-eigenvalues}, $\lambda_0$ is simple  and all the other eigenvalues are
semisimple, and their multiplicity can be computed explicitly.

We are going to prove the following results.

\begin{theorem}\label{stab-infected}
Under the hypothesis \eqref{hyp-Tmax}, it holds:
\begin{enumerate}[\rm (i)]
\item the infected equilibrium ${\mathbf X}_i= (T_i,I_i,V_i)$ is asymptotically stable for
$(N,r) \in {\mathcal I}\setminus {\mathcal P}$, where
${\mathcal P}$ is defined in \eqref{set-Pk} with $k=0$;
\item the infected equilibrium ${\mathbf X}_i= (T_i,I_i,V_i)$   is unstable in the interior of ${\mathcal P}$.
\end{enumerate}
\end{theorem}

\subsection{The resolvent equation}
As in the proof of Theorem \ref{stab-uninf}, we do not stress explicitly the dependence on $r$ and $N$ of the operators and
the sets that we consider in what follows. For $(N,r) \in \mathcal I$,
the linearization around ${\mathbf X}_i=(T_i,I_i,V_i)$
of System \eqref{pde1}-\eqref{pde3} is associated with the
linear operator
\begin{align}\label{jacobian-inf}
\boldsymbol{{\mathcal L}}_i=
\begin{pmatrix}
 -\left (\frac{\mv r}{\gamma NT_{\max}}+\frac{\alpha\gamma N}{\mv}\right )Id & 0 & -\frac{\mv}{N}Id \\[2mm]
 \left [\frac{\alpha\gamma N}{\mv}-\mt+r\left (1-\frac{\mv}{\gamma N T_{\max}}\right )\right ]Id & -\mi Id  & \frac{\mv}{N}Id \\[2mm]
    0 & N \mi Id  & d_V\Delta-\mv Id
\end{pmatrix}.
\end{align}

\begin{proposition} \label{spectrum Li}
For any $(N,r)\in {\mathcal I}$ the realization $\boldsymbol{L}_i$
of the operator $\boldsymbol{{\mathcal L}}_i$ in $(L^2_{\mathbb C})^3$ with domain
$D(L_i)=L^2_{\mathbb C}\times L^2_{\mathbb C}\times H^2_{\sharp,\mathbb C}$ generates an analytic
strongly continuous semigroup. Moreover, the spectrum of $\boldsymbol{L}_i$ is given by
\begin{equation}
\sigma (\boldsymbol{L}_i)=
\left\{-\frac{\mv r}{\gamma NT_{\max}}
-\frac{\alpha\gamma N}{\mv},-\mi\right\}\cup
\bigcup_{k\in\N}
\sigma_k,
\label{spectrum-k}
\end{equation}
where, for any $k\in\N$, $\sigma_k$ is the spectrum of the matrix
\begin{eqnarray*}
M_k=\begin{pmatrix}
 -\frac{\mv r}{\gamma NT_{\max}}-\frac{\alpha\gamma N}{\mv} & 0 & -\frac{\mv}{N} \\[2mm]
  \frac{\alpha\gamma N}{\mv}-\mt+r\left (1-\frac{\mv}{\gamma N T_{\max}}\right ) & -\mi  & \frac{\mv}{N} \\[2mm]
    0 & N \mi  & -d_V\lambda_k-\mv  \\
\end{pmatrix}.
\end{eqnarray*}
\end{proposition}
\begin{proof}
The same arguments as in the proof of Theorem \ref{stab-uninf}
show that $\boldsymbol{L}_i$ generates an analytic
strongly continuous semigroup in $(L^2_{\mathbb C})^3$.

Let us determine its
spectrum. For this purpose we use the discrete Fourier transform.
If a function ${\bf v}=(v_1,v_2,v_3)$ in $L^2_{\mathbb C}\times L^2_{\mathbb C}\times
H^2_{\sharp,\mathbb C}$ solves the resolvent equation
$\lambda {\bf v}-\boldsymbol{L}_i{\bf v}={\bf f}$, for some $\lambda\in\C$ and ${\bf f}=(f_1,f_2,f_3)$ in
$(L^2_{\mathbb C})^3$, then its Fourier coefficients ${\bf v}_k=(v_{1,k},v_{2,k},v_{3,k})$
($k=0,1,\ldots$) solve the infinitely many equations
$(\lambda Id-M_k){\bf v}_k={\bf f}_k$ ($k=0,1,\ldots$),
where ${\bf f}_k=(f_{1,k},f_{2,k},f_{3,k})$ and $f_{j, k}$ denotes the $k$-th Fourier coefficient of the
function $f_j$ ($j=1,2,3$). Clearly, any eigenvalue of $M_k$
($k=0,1,\ldots$) is an eigenvalue of $\boldsymbol{L}_i$. Therefore,
$\sigma (\boldsymbol{L}_i)\supset \bigcup_{k\in\N} \sigma_k$.

On the other hand, if $\lambda\not\in\sigma_k$ for any
$k=0,1,\ldots$, then all the coefficients
$(v_{1,k},v_{2,k},v_{3,k})$ are uniquely determined through the
formulae
\begin{align}
v_{1,k}=&\frac{1}{{\mathcal D}_k(\lambda)}\left\{
[(\lambda+\mi)(\lambda+d_V\lambda_k+\mv)-\mi\mv]f_{1,k}
-\mi\mv f_{2,k}-\frac{\mv}{N}(\lambda+\mi)f_{3,k}\right\},
\label{A}\\[1mm]
v_{2,k}=&\frac{1}{{\mathcal D}_k(\lambda)}\bigg\{
\frac{\lambda+d_V\lambda_k+\mv}{\gamma\mv N T_{\max}}
[(-\mv^2r+\alpha\gamma^2 N^2T_{\max}-\gamma\mt\mv N T_{\max} +\gamma\mv rN T_{\max})f_{1,k}\notag\\[1mm]
&\qquad\quad\qquad\qquad\qquad\;\, +
(\lambda\gamma\mv NT_{\max}+\mv^2r+\alpha\gamma^2 N^2T_{\max})f_{2,k}]\notag\\
&\qquad\quad\;+\frac{\mv}{\gamma N^2T_{\max}}(2\mv r-\gamma rN T_{\max}+\gamma\mt  N T_{\max}
+\lambda\gamma N T_{\max})f_{3,k}\bigg\},
\label{B}
\\[2mm]
v_{3,k}=&\frac{1}{{\mathcal D}_k(\lambda)}
\bigg\{\frac{\mi}{\gamma\mv T_{\max}}[(\alpha\gamma^2 N^2T_{\max}-\gamma\mt\mv NT_{\max}+\gamma\mv rN T_{\max}-\mv^2r)f_{1,k}\notag\\
&\qquad\qquad\qquad\quad+(\lambda\gamma\mv NT_{\max}+\mv^2r+\alpha\gamma^2 N^2T_{\max})f_{2,k}]\notag\\
&\qquad\quad+\frac{\lambda+\mi}{\gamma\mv NT_{\max}}
(\lambda\gamma\mv  N T_{\max}+\mv^2r+\alpha\gamma^2 N^2T_{\max})f_{3,k}\bigg\},
\label{C}
\end{align}
where
\begin{equation}
{\mathcal D}_k (\lambda)= \lambda^3 + d_{1,k} \lambda^2 + d_{2,k} \lambda +
d_{3,k},\qquad\;\,k=0,1, \ldots,
\label{Dklambda}
\end{equation}
and
\begin{align}
d_{1,k}&=d_{1,k}(N,r)=d_V\lambda_k+\mi+\mv+\frac{\mv r}{\gamma NT_{\max}}+\frac{\alpha\gamma N}{\mv},
\label{d1k} \\
d_{2,k}&=d_{2,k}(N,r)= \mi d_V\lambda_k+\alpha\gamma N+\frac{\alpha\gamma N}{\mv}(\mi+d_V\lambda_k)
+\frac{\mv r}{\gamma NT_{\max}}(\mi+\mv+d_V\lambda_k),
\label{d2k} \\
d_{3,k}&=d_{3,k}(N,r)= \mi\mv(r-\mt)+\frac{\alpha\gamma\mi N}{\mv}(\mv+d_V\lambda_k)
+\frac{\mi\mv r}{\gamma NT_{\max}}(d_V\lambda_k-\mv).
\label{d3k}\end{align}

Note that, if $\lambda$ differs from both $-\frac{\mv r}{\gamma
NT_{\max}}-\frac{\alpha\gamma N}{\mv}$ and $-\mi$, then
\begin{eqnarray*}
{\mathcal D}_k(\lambda)\sim
d_V\left\{\lambda^2+\lambda\left (\mi+\frac{\alpha\gamma N}{\mv}
+\frac{\mv r}{\gamma N T_{\max}}\right )
+\frac{\alpha\gamma\mi N}{\mv}+\frac{\mi\mv r}{\gamma NT_{\max}}\right\}\lambda_k,
\end{eqnarray*}
as $k\to +\infty$. Hence, for any $\lambda\notin\left\{-\frac{\mv r}{\gamma
NT_{\max}}-\frac{\alpha\gamma N}{\mv},-\mi\right\}\cup\bigcup_{k\in\N}\sigma_k$, it holds that
\begin{align*}
v_{1,k} &\sim
\frac{\lambda+\mi}{\lambda^2+\lambda\left (\mi+\frac{\alpha\gamma N}{\mv}
+\frac{\mv r}{\gamma N T_{\max}}\right )
+\frac{\alpha\gamma\mi N}{\mv}+\frac{\mi\mv r}{\gamma NT_{\max}}}f_{1,k},\\[2mm]
v_{2,k}&\sim \frac{\left (-\frac{\mv r}{\gamma NT_{\max}}
+\frac{\alpha\gamma N}{\mv}+r-\mt\right )f_{1,k}
+\left (\frac{\mv r}{\gamma NT_{\max}}+\frac{\alpha\gamma N}{\mv}
+\lambda\right )f_{2,k}}{\lambda^2+\lambda\left (\mi
+\frac{\alpha\gamma N}{\mv}+\frac{\mv r}{\gamma N T_{\max}}\right )
+\frac{\alpha\gamma\mi N}{\mv}+\frac{\mv\mi r}{\gamma NT_{\max}}},\\[2mm]
v_{3,k} &\sim\frac{1}{d_V\lambda_k\left\{\lambda^2
+\lambda(\mi+\frac{\alpha\gamma N}{\mv}+\frac{\mv r}{\gamma N T_{\max}})+\frac{\alpha\gamma\mi N}{\mv}
+\frac{\mi\mv r}{\gamma NT_{\max}}\right\}}\\
&\qquad\quad\times \bigg\{ \bigg (\frac{\alpha\gamma\mi N^2}{\mv}
-\frac{\mi\mv r}{\gamma T_{\max}}+\mi rN-\mi\mt N\bigg )f_{1,k}
+\bigg (\frac{\alpha\gamma \mi N^2}{\mv}+\frac{\mi\mv r}{\gamma T_{\max}}+\lambda\mi N\bigg )f_{2,k}\\
&\qquad\qquad\quad +(\lambda+\mi)\bigg (\lambda+\frac{\alpha\gamma N}{\mv}
+\frac{\mv r}{\gamma NT_{\max}}\bigg )f_{3,k}\bigg\},
\end{align*}
as $k\to +\infty$. Thus, the sequences $\{v_{1,k}\}$,
$\{v_{2,k}\}$ and $\{\lambda_kv_{3,k}\}$ are square-summable. This shows
that the series whose Fourier coefficients are $v_{1,k}$, $v_{2,k}$
and $v_{3,k}$, respectively, converge in $L^2_{\mathbb C}$ (the first two ones)
and in $H^2_{\sharp,\mathbb C}$ (the latter one). The inclusion
$\sigma\subset \left\{-\frac{\mv r}{\gamma
NT_{\max}}-\frac{\alpha\gamma N}{\mv},-\mi\right\}\cup
\bigcup_{k\in\N}\sigma_k$
follows.
On the other hand, the previous computations show that, if $\lambda=-\mi$ or $\lambda=-\frac{\mv r}{\gamma NT_{\max}}-\frac{\alpha\gamma
N}{\mv}$, then the series having $v_{1,k}$, $v_{2,k}$ and $v_{3,k}$ as Fourier coefficients do not, in general, converge
in $L^2$ (the first two ones) and in $H^2_{\sharp}$ (the latter one). Hence, these values of $\lambda$ belong to the
essential spectrum of $\boldsymbol{L}_i$. Thus, \eqref{spectrum Li} is proved.
\end{proof}

\subsection{Study of $\sigma_k$}
\label{subsect-4.2}
Clearly, at fixed $k=0,1, \ldots$, each set $\sigma_k$ consists of at most three eigenvalues $\nu_{j,k}, j=1,2,3$,
either all real, or one real and two complex
conjugates, which verify the equation
$\lambda^3 + d_{1,k} \lambda^2 + d_{2,k} \lambda + d_{3,k}= 0$,
with $d_{j,k}$ $(j=1,2,3)$  being given by \eqref{d1k}-\eqref{d3k}.

The Routh-Hurwitz criterion enables us to determine whether
the elements of $\sigma_k$  have negative real parts. The latter
holds if and only if $d_{1,k}$, $d_{3,k}$ and the leading Hurwitz determinant
$D_{2,k}=d_{1,k} d_{2,k}-d_{3,k}$ are positive. The case $k=0$
corresponds to the system of ODEs  considered in \cite{FB10}.

Obviously, $d_{1,k}>0$. As far as $d_{3,k}$ is concerned, we remark
that $d_{3,k}(N,r)>d_{3,0}(N,r)$, which is positive in ${\mathcal I}$
and vanishes at $N_{\rm crit}(r)$.

For $(N,r) \in {\mathcal I}$, we compute the Hurwitz determinant $D_{2,k}(N,r)$ and we get
\begin{equation}
D_{2,k}(N,r)=\frac{1}{\gamma^2\mv^2N^2T_{\max}^2}\left(A_kr^2+B_k(N)r+C_k(N)\right),
\label{D2k}
\end{equation}
where
\begin{align*}
&A_k= \mv^4(\mi+\mv+d_V\lambda_k),\\[2mm]
&B_k(N)=\gamma\mv^2 NT_{\max} \Big [
\mv d_V^2\lambda_k^2+2\alpha\gamma d_V N\lambda_k+2\mv^2d_V\lambda_k+2\mi\mv d_V\lambda_k\\
  &\qquad\qquad\qquad\qquad\;
-\gamma\mi\mv NT_{\max}+2\alpha\gamma\mi N+2\alpha\gamma \mv N+\mi^2\mv+3\mi\mv^2+\mv^3\Big ],\\[2mm]
&C_k(N)={N}^{2}{\gamma}^{2}T_{\max}^{2}\Big (\mi\mv^2d_V^2\lambda_k^2+\alpha\gamma\mv d_V^2N\lambda_k^2
+\alpha^2\gamma^2d_VN^2\lambda_k+2\alpha\gamma\mv^2d_VN\lambda_k\\
&\qquad\qquad\qquad\quad\;\;\;\;
+\mi\mv^3d_V\lambda_k+\mi^2\mv^2d_V\lambda_k+2\alpha\gamma\mi\mv d_VN\lambda_k
+\alpha\gamma \mv^3N+\alpha\gamma\mi\mv^2N\\
&\qquad\qquad\qquad\quad\;\;\;\;
+\alpha\gamma \mi^2\mv N+\alpha^2\gamma^2\mv N^2+\alpha^2\gamma^2\mi N^2+\mi\mt\mv^3\Big ).
\end{align*}
Both $A_k$ and $C_k(N)$ are positive, whereas $B_k(N)$ vanishes at
\begin{eqnarray}\label{widehat-N}
N_{0,k}=-\frac{\mv(d_V^2\lambda_k^2
+2(\mi+\mv)d_V\lambda_k+3\mi\mv+\mi^2+\mv^2)}{\gamma(2\alpha d_V\lambda_k+2\alpha\mi+2\alpha\mv-\mi\mv T_{\max})}.
\end{eqnarray}
\begin{figure}[htp]
\centering
\includegraphics[width=9cm,height=6cm]{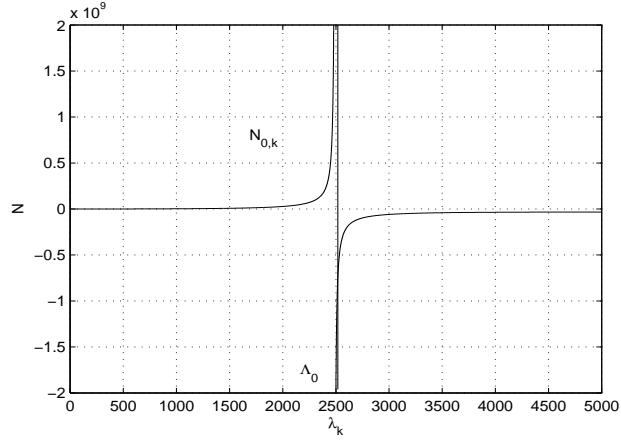}
\vskip -.4truecm
\caption{Profile of the curve $N_{0,k}$.}
\label{N_0_k}
\end{figure}

In \eqref{widehat-N}, as a function of $\lambda_k$, the denominator vanishes at
\begin{eqnarray*}
\Lambda_0=\frac{\mi\mv T_{\max}}{2\alpha d_V}-\frac{\mi+\mv}{d_V},
\end{eqnarray*}
which is positive and generically does not meet any of the $\lambda_k$'s, $k\ge 1$.
There are two cases (see Fig. \ref{N_0_k}):
\begin{enumerate}[(i)]
\item
$0 \leq \lambda_k<\Lambda_0$, hence $N_{0,k}>0$. Then,
$B_k(N)>0$ if $0 < N<N_{0,k}$;
\item
$\lambda_k>\Lambda_0$, hence $N_{0,k}<0$ and in this case
$B_k(N)$ is positive for any $N>0$.
\end{enumerate}

The sign of the polynomial
$A_kr^2+B_k(N)r+C_k(N)$ is obviously related to the discriminant
\begin{eqnarray*}
\Delta_k(N) =
(B_k(N))^2-4A_k C_k(N)=\gamma^{2}\mv^5N^2 T_{\max}^2(a_k N^2+b_kN+c_k),
\end{eqnarray*}
which in turn is of the sign of
$a_k N^2+b_k N+c_k $.
The coefficients $a_k$, $b_k$ and $c_k$
read:
\begin{align*}
a_k =&{\gamma}^2\mi T_{\max} \Big\{\mi\mv T_{\max}-4\alpha d_V\lambda_k
-4\alpha(\mi + \mv) \Big\},\\
b_k =&-2 \gamma \mi\mv \Big\{[d_V^2\lambda_k^2 +2(\mi+\mv)d_V\lambda_k+\mi^2+3 \mi\mv+\mv^2]T_{\max} - 4\alpha d_V\lambda_k
-4\alpha(\mi+\mv)\Big\},\\
c_k =& \mv\Big\{(d_V^2\lambda_k^2-\mi^2)^2
+4\mv d_V^3\lambda_k^3+6\mv(\mi+\mv)d_V^2\lambda_k^2\\
&\quad\;+4\mv[\mv^2+3\mi\mv+\mi(2\mi-\mt)]d_V\lambda_k\\
&\quad\;+2\mi^2\mv(3\mi-2\mt)+6\mi\mv^3+\mi\mv^2(11\mi-4\mt)+\mv^4\Big\}.
\end{align*}
Let us examine the signs of these coefficients.
\begin{enumerate}[(i)]
\item
As  a function of $\lambda_k$, the coefficient $a_k$ vanishes at
\begin{eqnarray*}
\Lambda_2=\frac{\mi\mv T_{\max}}{4\alpha d_V}- \frac{\mi+\mv}{d_V}
= \Lambda_0 - \frac{\mi\mv T_{\max}}{4\alpha d_V}.
\end{eqnarray*}
Clearly, $\Lambda_2$ is positive thanks to \eqref{hyp-Tmax} and, as $\Lambda_0$, generically
does not meet any of the $\lambda_k$'s for $k\geq 0$.
Then, $a_k$ is positive if $0\leq \lambda_k<\Lambda_2$ and negative otherwise.
\item
$b_k<0$ due the hypothesis \eqref{hyp-Tmax}.
\item
$c_k>0$ under the biologically relevant hypothesis
$\mi>\mt$ (see \eqref{bio-hypoth-1}).
\end{enumerate}
Next we compute:
\begin{align*}
\delta_k := &b_k^2-4a_kc_k\\
=&16 {\gamma}^{2}\mi\mv ( \mi+\mv+d_V\lambda_k) \\
&\;\;\times
 \Big\{\alpha (d_V^2\lambda_k^2-\mi^2)^2T_{\max}
+4\alpha\mv d_V^3 T_{\max}\lambda_k^3
+(\mi^2\mv T_{\max}+6\alpha\mv^2+4\alpha\mi\mv)d_V^2T_{\max}\lambda_k^2\\
&\quad\quad +
\mv\big [\mi^2(\mi +\mv) T_{\max}^2+4\alpha(\mi^2+2\mi\mv+\mv^2-\mi\mt)T_{\max}
+4\alpha^2\mi \big ]
d_V\lambda_k\\
&\quad\quad+\mi^2\mt\mv^2 T_{\max}^2
+\alpha[\mi^4+4 \mi\mv^3+\mv^4
+4\mi^2\mv(\mi-\mt)+\mi\mv^2(5\mi-4\mt)]T_{\max}\\
&\quad\quad
+4\alpha^2\mi^2\mv+4\alpha^2\mi\mv^2\Big\}.
\end{align*}
Again, thanks to the hypothesis $\mi>\mt$, $\delta_k$ is
always positive. Therefore, the roots of
$a_kN^2+b_kN+c_k =0$, namely the ones of $\Delta_k(N)=0$,  are:
\begin{eqnarray}\label{N1-N2}
N_{1,k}=\frac{-b_k-\sqrt{\delta_k}}{2a_k}, \qquad\;\,
N_{2,k}=\frac{-b_k+\sqrt{\delta_k}}{2a_k}.
\end{eqnarray}
\begin{remark}\label{position N0k}
Since $B_k(N_{0,k})=0$ and $A_k(N_{0,k})$, $C_k(N_{0,k})$ are both positive, $a_k(N_{0,k}))^2+b_kN_{0,k}+c_k <0$.
This provides us with a useful information regarding the position of $N_{0,k}$
with respect to $N_{1,k}$ and $N_{2,k}$, according to
\eqref{N1-N2}, i.e., $N_{1,k}<N_{0,k}<N_{2,k}$ if $a_k>0$ and
$N_{0,k}<N_{2,k}$ or $N_{0,k}>N_{1,k}$ if $a_k<0$.
 In particular, $N_{0,k}$ can meet neither $N_{1,k}$ nor $N_{2,k}$.
\end{remark}

Now we are in a position to begin the discussion, depending on the
position of $\lambda_k$.
Let us distinguish three cases.

\begin{figure}[htp]
\centering
\includegraphics[width=9cm,height=6cm]{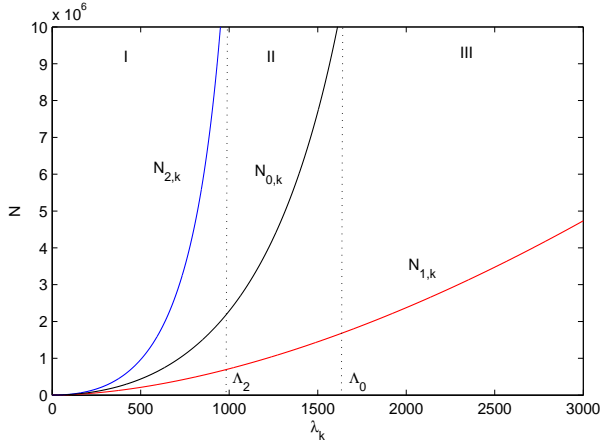}
\vskip -.4truecm
\caption{Comparison of the three curves $N_{0,k}$, $N_{1,k}$ and $N_{2,k}$,
as functions of $\lambda_k$.}
\label{lambda_k vs N}
\end{figure}

\paragraph{\bf Case I: $\boldsymbol{0\leq \lambda_k<\Lambda_2}$.}
In this situation $a_k>0$, $b_k<0$ and $c_k>0$ and $k$ ranges in a finite set of indexes.
It is an extension of the case $k=0$, see \cite{FB10}. It follows from Remark \ref{position N0k}
that $B_k$ vanishes between $N_{1,k}$ and $N_{2,k}$ which are both positive.
In particular, for $T_{\max}$ large enough (as we are assuming), $N_{1,k}>\mv/(\gamma T_{\max})$.
Indeed,
\begin{align*}
&a_k\left (\frac{\mv}{\gamma T_{\max}}\right )^2+b_k\frac{\mv}{\gamma T_{\max}}+c_k\\
=&
\mv(d_V^2\lambda_k^2-\mi^2)^2+4\mv^2d_V^3\lambda_k^3+4\mi\mv^2d_V^2\lambda_k^2+6\mv^3d_V^2\lambda_k^2+8\mi\mv^3 d_V\lambda_k\\
&+4\mi\mv^2(\mi-\mt) d_V\lambda_k
+4\mv^4 d_V\lambda_k+2\mi^3\mv^2+2\mi^2\mv^3+2\mi^2\mv^2(\mi-\mt)\\
&+4\mi\mv^4+4\mi\mv^3(\mi-\mt)
+\mv^5+o(1),
\end{align*}
as $T_{\max}\to +\infty$.
Again, since $\mi>\mt$,
\begin{eqnarray*}
a_k\left (\frac{\mv}{\gamma T_{\max}}\right )^2+b_k\frac{\mv}{\gamma T_{\max}}+c_k>0,
\end{eqnarray*}
if $T_{\max}$ is large enough. It follows that either $\mv/(\gamma T_{\max})<N_{1,k}$ or $\mv/(\gamma T_{\max})>N_{2,k}$.
But as it is immediately seen, $N_{2,k}>\mv/(\gamma T_{\max})$. Indeed,
\begin{align*}
N_{2,k}\sim &\frac{d_V^2\lambda_k^2+2(\mi+\mv)d_V\lambda_k+\mi^2+3\mi\mv+\mv^2}{\gamma\mi T_{\max}}\\
&+\frac{\sqrt{\mv d_V^2\lambda_k^2+
+(\mv^2+\mi\mv)d_V\lambda_k+\mt\mv^2}}{2\gamma^2\mi\mv T_{\max}}\\
\ge &\frac{3\mv}{\gamma T_{\max}},
\end{align*}
as $T_{\max}\to +\infty$. Hence, $\mv/(\gamma T_{\max})<N_{1,k}$ as it has been claimed.

We consider four subcases depending on the position of $N$ with respect to
$N_{1,k}$ and $N_{2,k}$.
\begin{enumerate}[(i)]
\item
Assume $N_{1,k} <N <N_{2,k}$. Then,
$\Delta_k(N)<0$. Since $A_k>0$, it follows that $D_{2,k}(N,r)>0$ for all $r > r_{\rm crit}(N)$ (see \eqref{D2k} and recall
that we are taking $(N,r)$ from ${\mathcal I}$).
\item
If $N<N_{1,k}$, then $B_k(N)>0$. Hence, $A_kr^2+B_k(N)r+C_k(N)>0$ for
any $r\ge 0$ since $A_k,B_k,C_k>0$. It thus follows that $D_{2,k}(N,r)>0$ for all $r > r_{\rm crit}(N)$.
\item
If $N>N_{2,k}$, then $B_k(N)<0$. Since $\Delta_k(N)>0$, the equation
$A_kr^2+B_k(N)r+C_k(N)=0$ admits the two real and positive roots:
\begin{eqnarray}\label{r1r2}
r_{1,k}(N)=\frac{-B_k(N)-\sqrt{\Delta_k(N)}}{2A_k},\qquad\;\,
r_{2,k}(N)=\frac{-B_k(N)+\sqrt{\Delta_k(N)}}{2A_k}.
\end{eqnarray}
Observe that
\begin{align}
~~~~r_{1,k}(N)=&\frac{2C_k}{-B_k+\sqrt{\Delta_k}}\notag\\
\sim &\frac{1}{\mi\mv^3}\Big (\mi\mt\mv^3+\alpha\gamma\mv^3 N+\alpha\gamma \mi\mv^2N
+\alpha\gamma \mi^2\mv N+\alpha^2\gamma^2\mv N^2\notag\\
&\;\qquad\quad+\alpha^2\gamma^2\mi N^2+\alpha^2\gamma^2d_VN^2\lambda_k+2\alpha\gamma \mv^2d_VN\lambda_k
+\mi\mv^3d_V\lambda_k\notag\\
&\;\qquad\quad
+\mi\mv^2d_V^2\lambda_k^2+\alpha\gamma \mv d_V^2N\lambda_k^2+\mi^2\mv^2d_V\lambda_k+2\alpha\gamma\mi \mv d_VN\lambda_k\Big )\notag\\
&>\mt>\frac {(\mt\mv-\alpha\gamma N)^+}{\mv}=r_{\rm crit}(N),
\label{asympt-r1}
\end{align}
as $T_{\max}\to +\infty$.
Consequently, the Hurwitz determinant $D_{2,k}(N,r)$ is positive for
$r_{\rm crit}(N) \le r <r_{1,k}(N)$ and $r>r_{2,k}(N)$, it vanishes at
$r=r_{1,k}(N)$ and $r=r_{2,k}(N)$, and is negative for $r_{1,k}(N) <
r < r_{2,k}(N)$.
\item
Assume $N\in\{N_{1,k},N_{2,k}\}$. In such a case, $\Delta_k(N_{j,k})=0$ and
the polynomial $A_kr^2+B_k(N_{j,k})r+C_k(N_{j,k})$ has the double root
$r_k(N_{j,k})=-B_k(N_{j,k})/2A_k$. However, this solution makes sense only if $B_k(N_{j,k})<0$.
Hence, only the case $N=N_{2,k}$ is relevant, and we have $D_{2,k}(N_{2,k},r)>0$ for $r > r_{\rm crit}(N_{2,k})$
except at $r(N_{2,k})=-B_k(N_{2,k})/2A_k$, where it vanishes.
\end{enumerate}

We are now in a position to define the subdomain ${\mathcal P}_k$ of
${\mathcal I}$ by
\begin{equation}
{\mathcal P}_k = \left\{(N,r): N \geq N_{2,k},\,   r_{1,k}(N)
\leq r \leq r_{2,k}(N)\right\},
\label{set-Pk}
\end{equation}
see \eqref{N1-N2} and \eqref{r1r2},
and at $N=N_{2,k}$, $r_{1,k}(N_{2,k})=r_{2,k}(N_{2,k})= -B_k(N_{2,k})/2A_k$.
In the domain ${\mathcal I}$, $d_{1,k}$ and $d_{3,k}$
are positive, and $D_{2,k}(N,r)$ is positive except in
${\mathcal P}_k$. More precisely,  the Hurwitz determinant $D_{2,k}(N,r)$ is negative in the interior of
${\mathcal P}_k$ and it vanishes on the boundary of ${\mathcal P}_k$.

\paragraph{\bf Case II: $\boldsymbol{\Lambda_2 <\lambda_k<\Lambda_0}$.}
Now, $a_k<0$,
$b_k<0, c_k>0$. Hence, $N_{2,k}<0, N_{1,k}>0$. Since $N_{0,k}>0$, it holds that
$N_{0,k}>N_{1,k}$ according to Remark \ref{position N0k}. Moreover, as it is immediately seen,
$N_{0,k}>(\gamma T_{\max})^{-1}\mv$. There are
two possibilities:
\begin{enumerate}[(i)]
\item
$(N,r)\in {\mathcal I}$ satisfies
$N<N_{0,k}$. Then, $B_k(N)>0$. Therefore,
$A_kr^2+B_k(N)r+C_k(N)>0$ for any $r\ge r_{\rm crit}(N)$ since the coefficients
are all positive. It thus follows that $D_{2,k}(N,r)>0$.
\item
$(N,r)\in {\mathcal I}$ satisfies $N \geq N_{0,k}$. Then, $N>N_{1,k}$ and
$a_kN^2+b_kN+c_k <0$.
Therefore $\Delta_k(N)<0$ and $A_kr^2 +B_k(N)r +C_k(N)$ has the sign of
$A_k$ which is positive, so $D_{2,k}(N,r)>0$.
\end{enumerate}

\paragraph{\bf Case III: $\boldsymbol{\lambda_k > \Lambda_0}$.}
Here, $N_{0,k}<0$ and, therefore,  $B_k(N)>0$ for all $N>0$.
The conclusion is the same as in Case II (i).

We summarize our results in the following proposition.

\begin{proposition}\label{hurwitz}
Denote by
$K_2$ the largest integer such that
$\lambda_{K_2} < \Lambda_2$.
Then,
\begin{enumerate}[\rm (i)]
\item
for $k=0, \ldots, K_2$, the Hurwitz determinant $D_{2,k}(N,r)$ is, respectively, negative in the interior
of the subdomain ${\mathcal P}_k$ of ${\mathcal I}$, positive in ${\mathcal I }\setminus {\mathcal P}_k$, and it
vanishes on the boundary of ${\mathcal P}_k$;
\item
for $k=K_{2}+1, K_{2}+2, \ldots$, the Hurwitz determinant $D_{2,k}(N,r)$  is always positive in ${\mathcal I}$.
\end{enumerate}
\end{proposition}

\begin{remark}
To give an idea, with the numerical values of Table \ref{table1} and
$\ell=1$, $\Lambda_2 =1239.5$ and lies between
$\lambda_{97}=116\pi^2$ and $\lambda_{98}=128\pi^2$, therefore $K_2=97$.
\end{remark}

To conclude this subsection we prove the following proposition which gives a much clearer picture of how
the sets ${\mathcal P}_k$ are ordered in the space of the parameters.

\begin{proposition}\label{P inclusions} Let $K_2$ be as in the statement of Proposition \ref{hurwitz}.
Then, the following set inclusions hold:
\begin{eqnarray*}
{\mathcal P}_{K_2}\subseteq {\mathcal P}_{K_2-1}\subseteq\cdots \subseteq {\mathcal P}_{k}
\subseteq \cdots \subseteq {\mathcal P}_1\subsetneq {\mathcal P}_0.
\end{eqnarray*}
\end{proposition}

\begin{proof}
To begin with we claim that $N_{2,k}<N_{2,k+1}$ (see \eqref{N1-N2}) for any $k=0,\ldots,K_2-1$.
To prove the claim we observe that $\delta_k=\delta(\lambda_k)$, where
\begin{align*}
\delta(x)=&16 {\gamma}^{2}\mi\mv ( \mi+\mv+d_Vx)\\
&\;\;\times
 \Big\{\alpha d_V^4T_{\max}x^4
+4\alpha\mv d_V^3T_{\max}x^3+(\mi^2\mv T_{\max}+6\alpha\mv^2-2\alpha\mi^2+4\alpha\mi\mv)d_V^2T_{\max}x^2\\
&\quad\quad +
\mv \big [\mi^2(\mi +\mv) T_{\max}^2+4\alpha(\mi^2+2\mi\mv+\mv^2-\mi\mt)T_{\max}
+4\alpha^2\mi \big ]d_Vx\\
&\quad\quad
+\mi^{2}\mt\mv^{2} T_{\max}^{2}
+\alpha[\mi^4+4 \mi\mv^{3}+\mv^{4}
+4\mi^2\mv(\mi-\mt)+\mi\mv^2(5\mi-4\mt)]T_{\max}\\
&\quad\quad+4\alpha^2\mi^2\mv+4\alpha^2\mi\mv^2\Big\}.
\end{align*}
We compute the derivative of the function $\delta$ and get
\begin{align*}
\delta'(x)
=16\gamma^2\mi\mv\Big\{&5\alpha d_V^4T_{\max}x^4+4\alpha(\mi+5\mv)d_V^3T_{\max}x^3\\
&\,+3(\mi^2\mv T_{\max}+8\alpha\mi\mv-2\alpha\mi^2+10\alpha\mv^2)d_V^2T_{\max}x^2\\
&\,+4\big [\mi^2\mv(\mi+\mv)T_{\max}^2\\
&\;\;\;\;\;\;\;\;+\alpha(9\mi\mv^2+3\mi^2\mv-2\mi\mt\mv-\mi^3+5\mv^3)T_{\max}+2\alpha^2\mi\mv]d_Vx\\
&\,+\mi^2\mv(\mt\mv+ \mv^2+\mi^2+2\mi\mv)T_{\max}^2\\
&\,+\alpha(17\mi^2\mv^2-8\mi\mt\mv^2+8\mi^3\mv
+16\mi\mv^3+5\mv^4+\mi^4-8\mi^2\mt\mv)T_{\max}\\
&\,+8\alpha^2\mi^2\mv+8\alpha^2\mi\mv^2\Big\}d_V.
\end{align*}
Under hypothesis \eqref{hyp-Tmax} this function is positive and, consequently,
$k\mapsto \delta_k$ is non-decreasing.

Similarly, $a_k=a(\lambda_k)$ and $b_k=b(\lambda_k)$, the functions $a$ and $b$ being strictly decreasing.
Hence, the sequences $\{a_k\}$ and $\{b_k\}$ are non-increasing.
Since $a_k>0$ and $b_k<0$ we now easily get the claim.

To complete the proof of the inclusion ${\mathcal P}_{k+1}\subseteq {\mathcal P}_k$ for any
$k=0,\ldots,K_2-1$, we show that, for any $N\geq N_{2,k+1}$ we have
$r_{1,k}(N) \leq r_{1,k+1}(N) < r_{2,k+1}(N) \leq r_{2,k}(N)$.
These properties follow immediately from the definitions of $r_{1,k}(N)$ and $r_{2,k}(N)$ observing that
$0\le A_j\le A_{j+1}$, $B_j(N)\le B_{j+1}(N)\le 0$ (since $N\ge N_{0,h}$ for any $h=0,\ldots,k+1$;
recall that we are in the Case I(iii) where
$N_{1,k}$ and $N_{2,k}$ are both positive, and take Remark \ref{position N0k} into account)
and  $0\le C_j\le C_{j+1}$  for any $j=0,\ldots,K_2-1$.

Finally, since $\lambda_0<\lambda_1$, $N_{2,1}>N_{2,0}$. Consequently, ${\mathcal P}_1$
is properly contained in
${\mathcal P}_0$.
\end{proof}

\begin{figure}[ht]
\begin{center}
\begin{overpic}[width=10cm,height=7cm]{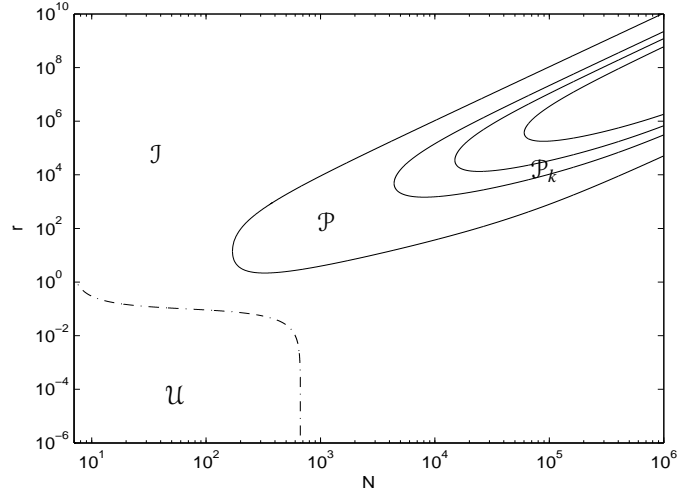}
\put(25,13){$\mathcal{U}$}
\put(23,45){$\mathcal{I}$}
\put(73,43){$\mathcal{P}_k$}
\put(45,36){$\mathcal{P}$}
\end{overpic}
\end{center}
\vskip -.4truecm
\caption{The sets ${\mathcal P}_k$.}
\label{graphPk}
\end{figure}

\subsection{Proof of Theorem \ref{stab-infected}}

The proof follows from Propositions \ref{hurwitz}, \ref{P inclusions}, Routh-Hurwitz criterion and the linearized
stability principle.

(i) As Proposition \ref{hurwitz} shows, for $k= 0, \ldots, K_2$ the leading Hurwitz
determinants $D_{2,k}(N,r)$ (see \eqref{D2k}) are positive in $\mathcal I \setminus {\mathcal P}_k$.
On the other hand, if $k \geq K_2+1$, then $D_{2,k}(N,r)>0$ for any $(N,r)\in {\mathcal I}$.

By Proposition \ref{P inclusions} it holds that ${\mathcal P}_k \subsetneq {\mathcal P}_0$, $k=1,2, \ldots$.
Hence, we conclude that $D_{2,k}(N,r)>0$ for any $k\in\N$, if $(N,r)\in {\mathcal I}\setminus {\mathcal P}_0$.
Since the other two Hurwitz determinants are positive in the whole of ${\mathcal I}$,
it follows from the Ruth-Hurwitz criterion that, if $(N,r) \in {\mathcal I} \setminus {\mathcal P}$,
then all the element of $\bigcup_{k\in\N}\sigma_k$ have negative real part. Hence, ${\rm Re}\,\sigma (\boldsymbol{L}_i) <0$
(see \eqref{spectrum-k}). It remains to invoke the linearized stability principle
as in the proof of Theorem \ref{stab-uninf}.

(ii) The instability of ${\bf X}_u$ can be deduced from \cite{FB10} which deals with
System \eqref{pde1}-\eqref{pde3} in the case when $d_V=0$ and shows that, in this situation, the infected
equilibrium ${\mathbf X}_i$ is unstable. \qed

\section{Hopf bifurcation and instability}\label{hopf-instability}
\setcounter{equation}{0}
For fixed $N>0$ we take the logistic parameter
$r>r_{\rm crit}(N)$ as a bifurcation parameter.

We recall that at fixed $(N,r)\in\mathcal I$, System \eqref{pde1}-\eqref{pde3}
has two equilibria: the uninfected
trivial solution ${\mathbf X}_u(N,r)$ and the infected, positive solution
${\mathbf X}_i(N,r)$. At ${\mathbf X}_i(N,r)$, the Jacobian matrix is
$\boldsymbol{{\mathcal L}}_i=\boldsymbol{{\mathcal L}}_{i,N,r}$, see \eqref{jacobian-inf}.
As we already observed in Proposition \ref{spectrum Li}, the realization $\boldsymbol{L}_{i,N,r}$ of
the operator $\boldsymbol{{\mathcal L}}_{i,N,r}$ in $(L^2_{\C})^3$ with domain
$D(\boldsymbol{L}_{i,N,r})=L^2_{\C}\times L^2_{\C}\times H^2_{\sharp,\C}$ generates an analytic
strongly continuous semigroup that we denote by $e^{t\boldsymbol{L}_{i,N,r}}$.

In this section we are interested in proving that Hopf bifurcation
occurs on the boundary of the set ${\mathcal P}$ (i.e.,
at the points $(N,r_{1,0}(N))$ and $(N,r_{2,0}(N))$ with $N\ge N_{2,0}$, where
$r_{1,0}(N)$, $r_{2,0}(N)$ and $N_{2,0}$ are given by \eqref{N1-N2} and \eqref{r1r2}) and in
analyzing the stability of the bifurcated periodic solutions.

Note, that for $T_{\max}$ large,
\begin{eqnarray*}
N_{2,0} = \frac{\mi^2+3\mi\mv+\mv^2+2\sqrt{\mi\mt\mv(\mi+\mv)}}{\gamma\mi }\,\frac{1}{T_{\max}}
 + o(T_{\max}^{-1}).
\end{eqnarray*}
Hence, $N_{2,0}$ is positive if $T_{\max}$ is large enough, let us say, if $T>T_{\max}^1(\mt,\mi,\mv,\alpha,\gamma)>T_{\max}^0(\mt,\mi,\mv,\alpha,\gamma)$ (see \eqref{hyp-Tmax}).
We assume hereafter that
\begin{eqnarray}\label{hyp-Tmax1}
T_{\max} \geq T_{\max}^1(\alpha,\gamma,\mi,\mt,\mv,N).
\end{eqnarray}

Here, differently from the previous sections, to avoid confusion we stress explicitly the
dependence of the operators, numbers and sets that we consider on $r$. We do not
stress the dependence on $N$ since in the following discussions only the parameter $r$ varies,
$N$ is (arbitrarily) fixed. In particular, we simply write $r_1$ and $r_2$ instead of $r_1(N)$ and $r_2(N)$.

\begin{theorem}\label{hopf2}
Let $N\ge N_{2,0}$ be fixed. Under the hypothesis \eqref{hyp-Tmax1}, Hopf bifurcation occurs at the
critical points $r=r_j$, $j=1,2$. More precisely,
\begin{enumerate}[\rm (i)]
\item
for any $\beta\in (0,1)$, there exist $c_0>0$ and smooth
functions $\hat r_j,\rho_j:(-\varepsilon_0,\varepsilon_0)\to\R$ and ${\bf X}_{\sharp,j}: (-\varepsilon_0,\varepsilon_0)\to
C^{1+\beta}(\R,L^2_{\sharp}\times L^2_{\sharp}\times L^2_{\sharp})\cap C^{\beta}(\R,L^2_{\sharp}\times L^2_{\sharp}\times H^2_{\sharp})$
$(j=1,2)$ such that $\rho_j(0)=1$, $\hat r_j(0)=r_j$, ${\bf X}_{\sharp,j}(0)={\bf X}_i(r_j)$, ${\bf X}_{\sharp,j}(c)$ is not constant in time if $c\neq 0$ and its period is $2\pi\rho_j(x)/\omega_j$, where
\begin{eqnarray*}
\omega=
\sqrt{\alpha\gamma N+\frac{\alpha\gamma\mi N}{\mv}
+\frac{\mv r_j}{\gamma NT_{\max}}(\mi+\mv)},\qquad\;\,j=1,2.
\end{eqnarray*}
\item
There exists $\delta>0$ such that, if ${\bf X}\in
C^{1+\beta}(\R,L^2_{\sharp}\times L^2_{\sharp}\times L^2_{\sharp})\cap C^{\beta}(\R,L^2_{\sharp}\times L^2_{\sharp}\times H^2_{\sharp})$
is a periodic solution to System \eqref{pde1}-\eqref{pde3} $($where $r$ is replaced by $\overline r)$
with period $2\pi\overline\rho/\omega$, such that
\begin{eqnarray*}
\;\;\;\;\;\;\;\;\;\;\|{\bf X}-{\bf X}_i(r_j)\|_{C^{\beta}(\R,L^2_{\sharp}\times L^2_{\sharp}\times H^2_{\sharp})}+
\|{\bf X}-{\bf X}_i(r_j)\|_{C^{1+\beta}(\R,L^2_{\sharp}\times L^2_{\sharp}\times L^2_{\sharp})}
+|\overline r-r_j|+|1-\overline\rho|\le\delta_0,
\end{eqnarray*}
for $j=0$ or $j=1$, then there exist $\varepsilon\in (-\varepsilon_0,\varepsilon_0)$ and
$t_0\in\R$ such that $\overline{r}=\hat r_j$ and ${\bf X}={\bf X}_{\sharp,j}(\varepsilon)$.
\end{enumerate}
\end{theorem}

\begin{proof}
We limit ourselves to considering the case when $r=r_1$, the case $r=r_2$ being completely similar.

For $r$ in some neighborhood of $r_1$, we set ${\mathbf u}={\mathbf X}-{\mathbf X}_i(r)$, $s=r-r_1$ and
write System \eqref{pde1}-\eqref{pde3} at the infected equilibrium as
\begin{equation}
\frac{d{\mathbf u}}{dt}=\boldsymbol{{\mathcal F}}({\mathbf u},s),
\label{eq-tilde-X}
\end{equation}
where
\begin{align*}
\boldsymbol{\mathcal F}_1({\bf u},s)=&
-\left (\frac{\mv (s+r_1)}{\gamma NT_{\max}}+\frac{\alpha\gamma N}{\mv}\right )u_1-\frac{\mv}{N}u_3
-\frac{r_1+s}{T_{\max}}u_1^2-\gamma u_1u_3;\\
\boldsymbol{\mathcal F}_2({\bf u},s)=&
\left [\frac{\alpha\gamma N}{\mv}-\mt+(s+r_1)\left (1-\frac{\mv}{\gamma N T_{\max}}\right )\right ]u_1
-\mi u_2+\frac{\mv}{N}u_3+\gamma u_1u_3;\\
\boldsymbol{\mathcal F}_3({\bf u},s)=&
\mi Nu_2+d_V\Delta u_3-\mv u_3
\end{align*}
Clearly, by the Sobolev embedding theorem, $\boldsymbol{\mathcal F}$ is a smooth function defined in $L^2\times L^2\times H^2_{\sharp}$.

Note that the derivative $\boldsymbol{\mathcal F}_{\bf u}({\bf 0},0)$ is the operator
$\boldsymbol{L}_{i,r_1}$ in Proposition \ref{spectrum Li}. More precisely,
\begin{align*}
\boldsymbol{L}_{i,r_1}=
\begin{pmatrix}
 -\left (\frac{\mv r_1}{\gamma NT_{\max}}+\frac{\alpha\gamma N}{\mv}\right )Id & 0 & -\frac{\mv}{N}Id \\[2mm]
 \left [\frac{\alpha\gamma N}{\mv}-\mt+r_1\left (1-\frac{\mv}{\gamma N T_{\max}}\right )\right ]Id & -\mi Id  & \frac{\mv}{N}Id \\[2mm]
    0 & N \mi Id  & d_V\Delta-\mv Id
\end{pmatrix}.
\end{align*}

By Proposition \ref{spectrum Li}, operator $\boldsymbol{L}_{i,r_1}$ is the generator of a strongly continuous analytic semigroup in $(L^2_{\C})^3$.

Let us prove that
$\sigma(\boldsymbol{L}_{i,r_1})$ consists of eigenvalues with negative real part and a pair of purely imaginary and conjugate eigenvalues
$\lambda_1(r_1)$ and $\lambda_2(r_1)$, which are simple eigenvalues and satisfy the transversality condition. Once checked, these properties will
yield the assertion in view of
\cite[Thm. 9.3.3]{lunardi1995analytic} (which deals with fully nonlinear problems but, of course, it applies also to
the semilinear case).

Being rather long, we split the proof into four steps.

{\em Step 1.} Here, we prove that $\sigma(\boldsymbol{L}_{i,r_1})$ consists of eigenvalues with negative real part and
a pair of purely imaginary and conjugate eigenvalues.
For this purpose, we observe that, since ${\mathcal P}_k$ is properly contained in ${\mathcal P}_0$
for any $k=1,\ldots,K_2$ (see Proposition \ref{P inclusions}), the pair $(N,r_1(N))$
belongs to ${\mathcal I}\setminus {\mathcal P}_k$ for any $k=1,\ldots,K_2$. Therefore, from the results in
Subsection \ref{subsect-4.2} and Proposition \ref{spectrum Li}, it follows that $\sigma_{k,r_1}$ is contained in the halfplane
$\{\lambda\in\C: \re\,\lambda<0\}$.

As far as $\sigma_0$ is concerned, Orlando formula (see e.g., \cite[Chpt. XV]{gantmakher}) shows that
the Hurwitz determinant $D_{2,0}(r)$ (see \eqref{D2k}) factorizes as follows:
\begin{eqnarray*}
D_{2,0}(r)=-(\lambda_1(r)+\lambda_2(r))(\lambda_2(r)+\lambda_3(r))(\lambda_1(r)+\lambda_3(r)),
\end{eqnarray*}
where $\lambda_1(r)$, $\lambda_2(r)$ and $\lambda_3(r)$ are the roots of the polynomial
\begin{eqnarray*}
{\mathcal D}_{0,r}(\lambda)=\lambda^3+d_{1,0}(r)\lambda^2+d_{2,0}(r)\lambda+d_{3,0}(r)
\end{eqnarray*}
(see \eqref{Dklambda}) (i.e. the elements of $\sigma_0$). The point $(r_1,N)$ lies on the boundary of
${\mathcal P}_0$. Hence, $D_{2,0}(r_1)$ vanishes, i.e.,
\begin{equation}
(\lambda_1(r_1)+\lambda_2(r_1))(\lambda_2(r_1)+\lambda_3(r_1))(\lambda_1(r_1)+\lambda_3(r_1))=0.
\label{orlando}
\end{equation}
Since the coefficients of ${\mathcal D}_{0,r_1}$ are real and positive, at least one of the three roots $\lambda_1(r_1)$, $\lambda_2(r_1)$, $\lambda_3(r_1)$ (let us say $\lambda_3(r_1)$)
is real and negative and the other two roots are either both negative or they are complex and conjugate. From \eqref{orlando} it follows
that $\lambda_1(r_1)$ and $\lambda_2(r_1)$ are purely imaginary and conjugate.

{\em Step 2.} Let us prove that there exists a gap between $\sigma(\boldsymbol{L}_{i,r_1})\setminus\{\lambda_1(r),\lambda_2(r)\}$ and the
imaginary axis.
We have to consider the set $\sigma_k=\sigma_{k,r_1}$ ($k=1,2,\ldots$) which consists of the roots of the third-order polynomial
${\mathcal D}_{k,r_1}$ (see \eqref{Dklambda}).
Indeed, as we have already remarked, $\lambda_3(r_1)$ is negative.

Write $\lambda=\mu-M$.
If $\tilde\lambda$ is a root of the polynomial ${\mathcal D}_{k,r_1}$, then $\tilde \mu=\tilde\lambda+M$
is a root of the polynomial
$p_{k,r_1}(\lambda)= \lambda^3 + \tilde d_{1,k}(r_1) \lambda^2 + \tilde d_{2,k}(r_1) \lambda + \tilde d_{3,k}(r_1)$, where
\begin{align*}
&\tilde d_{1,k}(r_1)=d_{1,k}(r_1)-3M,\\
&\tilde d_{2,k}(r_1)=d_{2,k}(r_1)-2Md_{1,k}(r_1)+3M^2,\\
&\tilde d_{3,k}(r_1)=d_{3,k}(r_1)-Md_{2,k}(r_1)+M^2 d_{1,k}(r_1)-M^3.
\end{align*}
As it is easily seen
\begin{align*}
&\tilde d_{1,k}(r_1)=d_V\lambda_k+o(\lambda_k),\\[1mm]
&\tilde d_{3,k}(r_1)=\left [\frac{\alpha\gamma\mi N}{\mv}+\frac{\mi\mv r_1}{\gamma NT_{\max}}-\left (\mi +\frac{\alpha\gamma N}{\mv}
+\frac{\mv r_1}{\gamma N T_{\max}}\right ) M
+M^2\right ]d_V\lambda_k+o(\lambda_k),\\[2mm]
&\tilde d_{1,k}(r_1)\tilde d_{2,k}(r_1)-\tilde d_{3,k}(r_1)=d_V^2\left (\mi +\frac{\alpha\gamma N}{\mv}+\frac{\mv r_1}{\gamma NT_{\max}}-2M\right )\lambda_k^2+o(\lambda_k^2),
\end{align*}
as $k\to +\infty$.
Hence, if we take $M$ satisfying the inequalities
\begin{align*}
\left\{
\begin{array}{l}
\displaystyle\frac{\alpha\gamma\mi N}{\mv}+\frac{\mi\mv r_1}{\gamma NT_{\max}}-\left (\mi+\frac{\alpha\gamma N}{\mv}
+\frac{\mv r_1}{\gamma N T_{\max}}\right ) M+M^2>0,\\[4mm]
\displaystyle\mi +\frac{\alpha\gamma N}{\mv}+\frac{\mv r_1}{\gamma NT_{\max}}-2M>0,
\end{array}
\right.
\end{align*}
then, for $k$ sufficiently large (say $k\ge K_3>K_2$), $\tilde d_{1,k}(r_1)$, $\tilde d_{3,k}(r_1)$
and $\tilde d_{1,k}(r_1)d_{2,k}(r_1)-\tilde d_{3,k}(r_1)$ are all positive.
Hence, Routh-Hurwitz criterion applies and shows that the roots of $p_{k,r_1}$ have negative
real part of any $k\ge K_3$. As a byproduct, $\bigcup_{k\ge K_3}\sigma_{k,r_1}\subset\{\lambda\in\C: \re\,\lambda<-M\}$.

Since $\bigcup_{1\le k<K_3}\sigma_{k,r_1}$ consists of finitely many eigenvalues with negative real part, up
to replacing $M$ with a smaller constant if needed, we can assume that
$\bigcup_{1\le k}\sigma_{k,r_1}\subset\{\lambda\in\C: \re\,\lambda<-M\}$.

{\em Step 3.} We now prove that the eigenvalues $\lambda_1(r_1)$ and $\lambda_2(r_1)$ are simple.

First, we prove that the resolvent operator $R(\lambda,\boldsymbol{L}_{i,r_1})$ has
a simple pole at $\lambda_j(r_1)$ ($j=1,2$). We limit ourselves to proving this property for the eigenvalue $\lambda_1(r_1)$, since for the other one the proof is completely similar.

From the proof of Proposition \ref{spectrum Li}, we know that, for any $\lambda\in\rho(\boldsymbol{L}_{i,r_1})$ and any
${\bf f}\in (L^2_{\C})^3$,
\begin{eqnarray*}
R(\lambda,\boldsymbol{L}_{i,r_1}){\bf f}=\left (
\sum_{k=0}^{+\infty}v_{1,k}(\lambda)\tilde e_k,\sum_{k=0}^{+\infty}v_{2,k}(\lambda)\tilde e_k,\sum_{k=0}^{+\infty}v_{3,k}(\lambda)\tilde e_k\right ),
\end{eqnarray*}
where $v_{j,k}(\lambda)$ $(j=1,2,3,\, k\in\N)$ are defined by \eqref{A}-\eqref{C}.

Observe that
\begin{align*}
{\mathcal D}_{k,r_1}(\lambda)=&\left [d_v\lambda^2+\left (\mi d_V+\frac{\alpha\gamma d_V N}{\mv}
+\frac{\mv r_1d_V}{\gamma N T_{\max}}\right )\lambda+\frac{\alpha\gamma\mi d_VN}{\mv}+\frac{\mi\mv r_1d_V}{\gamma NT_{\max}}\right ]\lambda_k+d(\lambda,r_1),
\end{align*}
where $d(\lambda,r_1)$ is independent of $k$, and it is smooth in $\lambda$.

As it is immediately seen the coefficient in front of $\lambda_k$ does not vanish at $\lambda=\lambda_1(r_1)$.
Hence, there exist a neighborhood $U$ of $\lambda_1(r_1)$, $k_0\in\N$ and a positive constant $\chi$ such that
$|{\mathcal D}_{k,r_1}(\lambda)|\ge \chi\lambda_k$ for any $k\ge k_0$ and any $\lambda$ in
$U$.
From \eqref{A}-\eqref{C} we thus deduce that
\begin{eqnarray*}
|v_{j,k}(\lambda)|\le C(|f_{1,k}|+|f_{2,k}|+|f_{3,k}|),
\end{eqnarray*}
for any $k\ge k_0$, $j=1,2,3$ and $\lambda\in U$.

Since ${\mathcal D}_{k,r_1}(\lambda_1(r_1))\neq 0$ for any $k\neq 0$, the previous estimate can be extended
to any $k\ge 1$.
Hence,
\begin{align*}
R(\lambda,\boldsymbol{L}_{i,r_1}){\bf f}=
\left (v_{1,0}(\lambda),v_{2,0}(\lambda),v_{3,0}(\lambda)\right )+\left (
\sum_{k=1}^{+\infty}v_{1,k}(\lambda)\tilde e_k,\sum_{k=1}^{+\infty}v_{2,k}(\lambda)\tilde e_k,\sum_{k=1}^{+\infty}v_{3,k}(\lambda)\tilde e_k\right ),
\end{align*}
where the second term in the previous splitting defines a function with values in ${\mathcal L}((L^2_{\C})^3)$ which is bounded in $U$.

The singularity of $R(\cdot,\boldsymbol{L}_{i,r_1})$ at $\lambda=\lambda_1(r_1)$ is due to the first term of the splitting.
The results in Step 1 show that $\lambda\mapsto D_{0,r_1}(\lambda)$ has a simple zero at $\lambda=\lambda_1(r_1)$.
It thus follows at once that the function $\lambda\mapsto ((\lambda-\lambda_1(r_1))v_{1,0}(\lambda),(\lambda-\lambda_1(r_1))v_{2,0}(\lambda),(\lambda-\lambda_1(r_1))v_{3,0}(\lambda))$ is bounded
around $\lambda=\lambda_1(r_1)$.

Summing up, we have proved that the function $\lambda\mapsto (\lambda-\lambda_1(r_1))R(\lambda,\boldsymbol{L}_{i,r_1})$ is
bounded around $\lambda=\lambda_1(r_1)$. Consequently,
$R(\cdot,\boldsymbol{L}_{i,r_1})$ has a simple pole at $\lambda=\lambda_1(r_1)$, so that, by \cite[Prop. A.2.2]{lunardi1995analytic} $\lambda_1(r_1)$ is a semisimple eigenvalue of $\boldsymbol{L}_{i,r_1}$.

To conclude that it is, actually, a simple eigenvalue, we have to show that the eigenspace associated with $\lambda_1(r_1)$ is one dimensional.
This property follows from recalling that ${\mathcal D}_{k,r_1}(\lambda_1(r_1))\neq 0$ if $k\ge 1$.
Hence, any eigenfunction associated with $\lambda_1(r_1)$ is a constant.

{\em Step 4 .} We now check the transversality condition.
Observing that
\begin{align*}
d_{1,0}(r_1)&=-(\lambda_1(r_1)+\lambda_2(r_1)+\lambda_3(r_1))=\lambda_3(r_1),\\
d_{2,0}(r_1)&=\lambda_1(r_1)\lambda_2(r_1)+\lambda_1(r_1)\lambda_2(r_3)+\lambda_1(r_2)\lambda_2(r_3)
=\lambda(r_1)\lambda(r_2),
\end{align*}
from
\eqref{d1k} and \eqref{d2k} we conclude that
\begin{eqnarray*}
\lambda_j(r_1)=(-1)^j\sqrt{\alpha\gamma N+\frac{\alpha\gamma\mi N}{\mv}
+\frac{\mv r_1}{\gamma NT_{\max}}(\mi+\mv)},\qquad\;\,j=1,2,
\end{eqnarray*}
and
\begin{eqnarray*}
\lambda_3(r_1)=-\mi-\mv-\frac{\mv r_1}{\gamma NT_{\max}}-\frac{\alpha\gamma N}{\mv}.
\end{eqnarray*}

By \cite[Chpt. 20]{K80} the function $r\mapsto\lambda_1(r)$ is smooth in a neighborhood of $r_1$.
Hence, differentiating the formula $(\lambda_1(r))^3+d_{1,0}(r)(\lambda_1(r))^2+d_{2,0}(r)\lambda_1(r)+d_{3,0}(r)=0$,
evaluating it at $r=r_1$ and then taking the real part, we get
\begin{eqnarray*}
\left (\frac{d}{dr}{\rm Re}\lambda_1\right )(r_1)
=\frac{d_{3,0}'(r_1)-d_{1,0}'(r_1)d_{2,0}(r_1)-d_{1,0}(r_1)d_{2,0}'(r_1)}{2(d_{2,0}(r_1)+(d_{1,0}(r_1))^2)}.
\end{eqnarray*}
The sign of $\left (\frac{d}{dr}{\rm Re}\lambda_1\right )(r_1)$ is the sign of
$d_{3,0}'(r_1)-d_{1,0}'(r_1)d_{2,0}(r_1)-d_{1,0}(r_1)d_{2,0}'(r_1)$. A straightforward computation shows that
\begin{align*}
&d_{3,0}'(r_1)-d_{1,0}'(r_1)d_{2,0}(r_1)-d_{1,0}(r_1)d_{2,0}'(r_1)\\
=&\frac{1}{\gamma^2N^2T_{\max}^2}\left [\gamma^2\mi\mt N^2T_{\max}^2
-\gamma NT_{\max}\left (3\mi\mv^2+2\alpha\gamma\mi N
+2\alpha\gamma\mv N+\mv^3+\mi^2\mv\right )\right.\\
&\qquad\qquad\;\;\,\left .-2\mv^3r_1-\mi\mv^2r_1\right ].
\end{align*}
Since
\begin{align*}
r_1\sim \frac{1}{\mi\mv^3}\Big (\mi\mt\mv^3+\alpha\gamma\mv^3 N+\alpha\gamma \mi\mv^2N
+\alpha\gamma \mi^2\mv N+\alpha^2\gamma^2\mv N^2+\alpha^2\gamma^2\mi N^2\Big ),
\end{align*}
as $T_{\max}\to +\infty$ (see \eqref{asympt-r1}), $\left (\frac{d}{dr}{\rm Re}\lambda_1\right )(r_1)$
is positive if $T_{\max}$ is sufficiently large, as we are assuming. Hence, the transversality condition is satisfied.
This completes the proof.
\end{proof}

\begin{proposition}
The bifurcated periodic solutions provided by Theorem \ref{hopf2} are independent of the spatial variables, i.e., they are
the same bifurcated periodic solutions of the following system of ODE's:
\begin{align}
\frac{\partial T}{\partial t}&=\alpha - \mt T+rT\left(1-\frac{T}{T_{\max}}\right) - \gamma V T,\label{ODE-1}\\
\frac{\partial I}{\partial t}&= \gamma V T - \mi I,\label{ODE-2}\\
\frac{\partial V}{\partial t}&=N \mi I- \mv V\label{ODE-3}.
\end{align}%
\end{proposition}

\begin{proof}
In \cite{FB10} it has been proved that System \eqref{ODE-1}-\eqref{ODE-3} exhibits a Hopf bifurcation at $r=r_j$ ($j=1,2$).
A branch of periodic solutions bifurcates from ${\bf X}_i(r_j)$ $(j=1,2)$.
Clearly, such solutions are space independent.
Moreover, a statement analogous to Theorem \ref{hopf2}(ii) holds for the Hopf bifurcation associated with
Problem \eqref{ODE-1}-\eqref{ODE-3}, see \cite[Thm. II, p. 16]{{hassard1981theory}}.
Therefore, up to replacing $\varepsilon_0$ with
a smaller value, if needed, we can infer that, for any $\varepsilon\in (-\varepsilon_0,\varepsilon_0)$,
${\bf X}_{\sharp}(\varepsilon)$ coincides, up to a translation in the time variable,
with one of the bifurcated periodic solutions in \cite[Thm. 4.5]{FB10}. This shows that
any function ${\bf X}_{\sharp}(\varepsilon)$ is space independent.
\end{proof}

We can now prove the following theorem:

\begin{theorem}\label{hopf4}
Suppose that $T_{\max}\ge T_{\max}^{(3)}$, where
$T_{\max}^{(3)}$ depends on $\alpha$, $\gamma$, $\mi$, $\mt$ and $\mv$ $($see the proof$)$. Then, the following properties are satisfied.
\begin{enumerate}[\rm (i)]
\item
If $N<N_*$ $($where $N_*$ is the first positive zero of the function ${\mathcal H}$ in \eqref{funct-N}$)$, then the periodic
solution ${\mathbf X}_{\#}^1(\e)$ is orbitally asymptotically stable with asymptotic phase.
\item
For any $N>0$, the periodic solution  ${\mathbf X}_{\#}^2(\e)$ is orbitally asymptotically stable with asymptotic phase.
\end{enumerate}
\end{theorem}

\begin{proof}
The arguments in Henry's book \cite{henry1981geometric} (see also \cite{daprato-lunardi} in a more general situation) show
that the stability of the bifurcated periodic solutions can be read on a Center Manifold.
This allows to reduce our problem, which is set in a infinite dimensional Banach space, to a problem in a finite dimensional space.

To obtain this finite dimensional problem, we first need to determine the spectral projection associated to the eigenvalues $-\omega_j i$ and
$\omega_ji$ ($j=1,2$).
As a general fact, such a projection is the sum of the spectral projections $P_{j,+}$,
associated to the eigenvalue $i\omega_j$, and
$P_{j,-}$, associated to the eigenvalue $-i\omega_j$. Since $i\omega_j$ and $-i\omega_j$ are simple eigenvalues (see Theorem
\ref{hopf2}), there exists a unique projection on the eigenspace relative to $i\omega_j$
which commutes with $\boldsymbol{L}_i$. Similarly, there exists a unique projection of the eigenspace relative
to $-i\omega_j$ which commutes with $\boldsymbol{L}_i$. Using these facts it is
easy to check that
\begin{eqnarray*}
P_j{\mathbf v}=P_{j,+}{\mathbf v}+ P_{j,-}{\mathbf v}=
\kappa_j\left (\int_{\Omega_{\ell}}{\mathbf v}\overline{\boldsymbol{\psi_j}}dxdy\right )\boldsymbol{\varphi_j}+
\overline{\kappa_j}\left (\int_{\Omega_{\ell}}{\mathbf v}\boldsymbol{\psi_j}dxdy\right )\overline{\boldsymbol{\varphi_j}},
\end{eqnarray*}
for any ${\bf v}\in (L^2_{\mathbb C})^3$, where
\begin{eqnarray*}
\boldsymbol{\varphi_j}=\left (\frac {\mv}{N(s_j-i\omega_j) },\frac{\mv+i \omega_j}{\mi N},1\right ),\qquad\;\,
\boldsymbol{\psi_j}
=\left ( -\frac {\xi_j}{s_j+i\omega_j},1,\frac{\mi-i\omega_j}{\mi N}\right )
\end{eqnarray*}
and
\begin{eqnarray*}
s_j=\displaystyle-\frac {\mv r_j}{\gamma NT_{\max}}-\frac {\alpha\gamma N}{\mv}, \qquad\;\,
\xi_j=\displaystyle\frac{\alpha\gamma N}{\mv}-\mt+\left (1-\frac{\mv}{\gamma NT_{\max}}\right )r_j,\qquad\;\,
\kappa^{-1}_j=(\boldsymbol{\varphi_j},\boldsymbol{\psi_j})_2,
\end{eqnarray*}
for $j=1,2$.
 In what follows we set
$\boldsymbol{\varphi_j}=(\varphi_{j,1},\varphi_{j,2},\varphi_{j,3})$ and
$\boldsymbol{\psi_j}=(\psi_{j,1},\psi_{j,2},\psi_{j,3})$.

As it is well known, $P_j$ allows to split $(L^2_{\C})^3$ into the direct sum of the two subspaces
$P_j((L^2_{\C})^3)$ and $(I-P_j)((L^2_{\C})^3)$ where
$P_j((L^2_{\C})^3)=\{z\boldsymbol{\varphi}+w\overline{\boldsymbol{\varphi}} : z,w\in\mathbb{C}\}$ and
$(I-P_j)((L^2_{\C})^3)=\{{\mathbf u}\in (L^2_{\C})^3: ({\mathbf v},\boldsymbol{\psi})_2=({\mathbf v},\overline{\boldsymbol{\psi}})_2=0\}$.
In particular, $P_j$ maps $(L^2)^3$ into itself and allows us to split the space $(L^2)^3$ into the direct sum of the two subspaces
$P_j((L^2)^3)=\{z\boldsymbol{\varphi}+\overline{z}\,\overline{\boldsymbol{\varphi}} : z\in\mathbb{C}\}$ and
$(I-P_j)((L^2)^3)=\{{\mathbf u}\in (L^2)^3: ({\mathbf u},\boldsymbol{\psi})_2=0\}$.

Let us rewrite Problem \eqref{eq-tilde-X} in the form
\begin{equation}
\frac{d{\mathbf u}}{dt}=\boldsymbol{L}_{i,r_j}{\mathbf u}+\boldsymbol{{\mathcal G}_j}({\mathbf u},s),
\label{eq-tilde-tilde-X}
\end{equation}
where
\begin{align*}
\boldsymbol{{\mathcal G}_{j,1}}({\bf u},s)=&
-\frac{\mv s}{\gamma NT_{\max}}u_1-\frac{r_j+s}{T_{\max}}u_1^2-\gamma u_1u_3;\\[1mm]
\boldsymbol{{\mathcal G}_{j,2}}({\bf u},s)=&
s\left (1-\frac{\mv}{\gamma N T_{\max}}\right )u_1
+\gamma u_1u_3;\\[1mm]
\boldsymbol{{\mathcal G}_{3,j}}({\bf u},s)=&0.
\end{align*}
Splitting Problem \eqref{eq-tilde-tilde-X} along $P_j((L^2)^3)$ and $(I-P_j)((L^2)^3)$, we see that any solution
${\bf u}\in C^1([0,a)\times (L^2)^3)\cap C([0,a)\times L^2\times L^2\times H^2_{\sharp})$
to Problem \eqref{eq-tilde-tilde-X}, defined in some
time domain $[0,a)$, can be identified with the pair of functions $(z,w)$, with $z(t)\in\C$
and ${\bf w}(t)\in L^2\times L^2\times H^2_{\sharp}$ for any $t\in [0,a)$, which solves the system
\begin{numcases}{ }
\frac{dz}{dt}=i\omega_j z+\tilde{{\mathcal G}_j}(z,\overline{z},{\bf w},s),
\label{dzdt} \\
\frac{d {\mathbf w}}{dt}=\boldsymbol{L}_i {\mathbf w}+\boldsymbol{{\mathcal K}_j}(z,\overline{z},{\mathbf w},s),
\label{dwdt}
\end{numcases}
where
\begin{align*}
\tilde{{\mathcal G}_j}(z,\overline{z},{\mathbf w},s)
=&\kappa_j\left (\gamma \varphi_{j,1}(1-\overline{\psi_{j,1}})-\frac{(r_j+s)\varphi_{j,1}^2\overline{\psi_{j,1}}}{T_{\max}}\right )\ell^2 z^2\\
&+\kappa_j\left (\gamma\overline{\varphi_{j,1}}(1-\overline{\psi_{j,1}})-\frac{(r_j+s)\overline{\varphi_{j,1}}^2\overline{\psi_{j,1}}}{T_{\max}}\right )\ell^2\overline{z}^2\\
&+2\kappa_j\left (\gamma(1-\overline{\psi_{j,1}}){\rm Re}\,(\varphi_{j,1})-\frac{(r_j+s)|\varphi_{j,1}|^2\overline{\psi_{j,1}}}{T_{\max}}\right )\ell^2 z\overline{z}\\
&
+\kappa_j\left [\int_{\Om_\ell}\left (\gamma(1-\overline{\psi_{j,1}})(\varphi_{j,1} w_3+w_1)-\frac{2(r_j+s)\varphi_{j,1}\overline{\psi_{j,1}}}{T_{\max}}w_1\right )dxdy\right.\\
&\qquad\;\,-\left.\frac{\mv s\varphi_{j,1}\overline{\psi_{j,1}}}{\gamma NT_{\max}}\ell^2+s\left (1-\frac{\mv}{\gamma N T_{\max}}\right )\varphi_{j,1}\ell^2\right ]z\\
&+\kappa_j\left [\int_{\Om_{\ell}}\left (\gamma(1-\overline{\psi_{j,1}})(\overline{\varphi_{j,1}}w_3+w_1)
-\frac{2(r_j+s)\overline{\varphi_{j,1}}\overline{\psi_{j,1}}}{T_{\max}}w_1\right )dxdy\right.\\
&\qquad\;\,-\left.\frac{\mv s\overline{\varphi_{j,1}}\overline{\psi_{j,1}}}{\gamma NT_{\max}}\ell^2+s\left (1-\frac{\mv}{\gamma N T_{\max}}\right )\overline{\varphi_{j,1}}\ell^2\right ]\overline{z}\\
&+\kappa_j\int_{\Om_\ell}\left (\gamma(1-\overline{\psi_{j,1}})w_3-\frac{(r_j+s)\overline{\psi_{j,1}}}{T_{\max}}w_1
-\frac{\mv s(\overline{\psi_{j,1}}+1)}{\gamma NT_{\max}}+s\right )w_1dxdy;\\[2mm]
\boldsymbol{{\mathcal K}_j}(z,\overline{z},{\mathbf w},s)=&\boldsymbol{{\mathcal G}_j}(z,\overline{z},{\mathbf w},s)
-2{\rm Re}\,(\tilde{{\mathcal G}_j}(z,\overline z, {\mathbf w},s)\boldsymbol{\varphi}),
\end{align*}
for $j=1,2$.
Modulo the identification of $P_j((L^2)^3)$ with the set $\{(z,\overline{z}): z\in\C\}$, the Center Manifold
for System \eqref{dzdt}-\eqref{dwdt} is the graph of a smooth function $\boldsymbol{\Upsilon_j}$ of the variable $(z,\overline z,s)$, defined in a neighborhood of zero with values
in $(I-P_j)((L^2)^3)$.

The equation to be analyzed, to understand the stability of the bifurcated solutions ${\bf X}_{\sharp}(\varepsilon)$, is therefore the following one:
\begin{equation}
\frac{dz}{dt}=i\omega_j z+\tilde{{\mathcal G}_j}(z,\overline{z},\boldsymbol{\Upsilon_j}(z,\overline{z},s),s)=:g_j(z,\overline{z},s).
\label{ODE}
\end{equation}
This ODE can be studied with classical methods (see e.g., \cite[Chpts. 1 \& 2]{hassard1981theory}).
One needs to expand the nonlinearity $g_j$ around $0$ as
\begin{eqnarray*}
g_j(z,\overline{z},0)=\sum_{2\leq h+k\leq3}\frac {g^{(j)}_{hk}}{h!k!}z_h\overline{z}_k+o(|z|^4).
\end{eqnarray*}
The coefficients $g_{hk}(s)$ are fundamental to determine the stability of the periodic solutions to \eqref{ODE}. In fact,
such solutions are stable if and only if $\re\, c_1(r_j)<0$, where (see e.g., \cite[p. 90]{hassard1981theory})
\begin{eqnarray*}
c_1(r_j)=\frac{i}{2\omega_j}\left (g_{20}^{(j)}g_{11}^{(j)}-2|g_{11}^{(j)}|^2-\frac{1}{3}|g_{02}^{(j)}|^2\right )+\frac{1}{2}g_{21}^{(j)}.
\end{eqnarray*}

To expand $g$ around the origin, one first needs to expand the function $\Upsilon(\cdot,\cdot,0)$ around $(0,0)$.
Since this function is smooth, we can expand it as
\begin{eqnarray*}
\boldsymbol{\Upsilon_j}(z,\overline{z},0)={\bf a}_1^{(j)}z+{\bf a_2}^{(j)}\overline{z}
+{\bf a_3}^{(j)}z^2+{\bf a_4}^{(j)}z\overline z+{\bf a_5}^{(j)}z^2+O(|z|^3).
\end{eqnarray*}
Replacing ${\bf u}(t)=z(t)\boldsymbol{\varphi}+\overline{z(t)}\overline{\boldsymbol{\varphi}}+\boldsymbol{\Upsilon}(z(t),\overline{z(t)})$ into System \eqref{dzdt}-\eqref{dwdt}, expanding
\begin{align*}
\boldsymbol{{\mathcal K}_j}(z,\overline{z},{\mathbf w},0)=\boldsymbol{{\mathcal K}_{j,1}}z^2+2\boldsymbol{{\mathcal K}_{j,2}}z\overline{z}+\overline{\boldsymbol{{\mathcal K}_{j,1}}}\,\overline{z}^2+O(|z||w|)+O(|w|^2),
\end{align*}
and observing that
\begin{eqnarray*}
\frac{d}{dt}{\bf w}(t)=\frac{\partial\boldsymbol{\Upsilon_j}}{\partial z}(z(t),\overline{z(t)})z'(t)
+\frac{\partial\boldsymbol{\Upsilon_j}}{\partial\overline{z}}(z(t),\overline{z(t)})\overline{z'(t)},
\end{eqnarray*}
an asymptotic analysis reveals that
\begin{eqnarray*}
\boldsymbol{\Upsilon_j}(z,\overline z,0)=z^2(2i\omega_j-\boldsymbol{L}_{i,r_j})^{-1}\boldsymbol{{\mathcal K}_{j,1}}-2z\overline{z}\boldsymbol{L}_{i,r_j}^{-1}\boldsymbol{{\mathcal K}_{j,2}}+\overline{z}^2(-2i\omega_j-\boldsymbol{L}_{i,r_j})^{-1}\overline{\boldsymbol{{\mathcal K}_{j,1}}}+O(|z|^3).
\end{eqnarray*}
where
\begin{align*}
&\boldsymbol{{\mathcal K}_{j,1}}=
\begin{pmatrix}
   \displaystyle \left(\frac{r_j\varphi_{j,1}^2}{T_{\max}}+\gamma \varphi_{j,1} \right)\left (2\ell^2{\rm Re}\,(\kappa_j\varphi_{j,1} \overline{\psi_{j,1}})-1\right )
   -2\gamma\ell^2 \varphi_{j,1}{\rm Re}\,(\kappa_j\varphi_{j,1})\\[3mm]
    \displaystyle 2\ell^2\left(\frac{\varphi_{j,1}^2 r_j}{T_{\max}}+\gamma \varphi_{j,1} \right){\rm Re}\,(\kappa_j\varphi_{j,2} \overline{\psi_{j,1}})
   + \gamma \varphi_{j,1}\left (1-2\ell^2{\rm Re}\,(\kappa_j\varphi_{j,2})\right )\\[3mm]
    \displaystyle 2\ell^2\left(\frac{\varphi_{j,1}^2 r_j}{T_{\max}}+\gamma \varphi_{j,1} \right){\rm Re}\,(\kappa_j\overline{\psi_{j,1}})
   -2\gamma\ell^2 \varphi_{j,1}{\rm Re}\,(\kappa_j)
   \end{pmatrix},
   \\[3mm]
&\boldsymbol{{\mathcal K}_{j,2}}=
\begin{pmatrix}
   \displaystyle \left (\frac{r_j|\varphi_{j,1}|^2}{T_{\max}}+\gamma {\rm Re}\,(\varphi_{j,1})\right )\left (2\ell^2{\rm Re}\,(\kappa_j\varphi_{j,1} \overline{\psi_{j,1}})-1\right )
   -2\gamma\ell^2{\rm Re}\,(\varphi_{j,1}){\rm Re}\,(\kappa_j\varphi_{j,1})\\[3mm]
    \displaystyle 2\ell^2\left (\frac{r_j |\varphi_{j,1}|^2}{T_{\max}}+ \gamma {\rm Re}\,(\varphi_{j,1})\right ){\rm Re}\,(\kappa_j\varphi_{j,2} \overline{\psi_{j,1}})
   +\gamma{\rm Re}\,(\varphi_{j,1})\left (1-2\ell^2{\rm Re}\,(\kappa_j\varphi_{j,2})\right )\\[3mm]
    \displaystyle 2\ell^2\left (\frac{r_j |\varphi_{j,1}|^2}{T_{\max}}+\gamma {\rm Re}\,(\varphi_{j,1})\right ){\rm Re}\,(\kappa_j\overline{\psi_{j,1}})
   -2\gamma\ell^2 {\rm Re}\,(\varphi_{j,1}){\rm Re}\,(\kappa_j)
\end{pmatrix}.
\end{align*}
Note that, since $\boldsymbol{{\mathcal K_j}}(z,\overline z,{\mathbf w})\in (I-P_j)((L^2)^3)$ for any $z\in\mathbb C$ and ${\bf w}\in (I-P_j)((L^2)^3)$,
$\boldsymbol{{\mathcal K}_{j,1}}$, $\boldsymbol{{\mathcal K}_{j,2}}$ and
$\overline{\boldsymbol{{\mathcal K}_{j1}}}$ belong to $(I-P_j)((L^2_{\C})^3)$ and the operators $\pm i\omega-\boldsymbol{L}_{i,r_j}$ are invertible
on $(I-P_j)((L^2_{\C})^3)$.

Now a long but straightforward computation shows that
\begin{align*}
g_{20}^{(j)}=&-\frac{2r_j\kappa_j\varphi_{j,1}^2\overline{\psi_{j,1}}\ell^2}{T_{\max}}+2\gamma\kappa_j \varphi_{j,1}(\overline{\psi_{j,2}}-\overline{\psi_{j,1}})\ell^2,\\
g_{11}^{(j)}=&-\frac{2r_j\kappa_j|\varphi_{j,1}|^2\overline{\psi_{j,1}}\ell^2}{T_{\max}}
+2\gamma\kappa_j\ell^2(\overline{\psi_{j,2}}-\overline{\psi_{j,1}}){\rm Re}(\varphi_{j,1}),\\
g_{02}^{(j)}=&-\frac{2r_j\overline{\varphi_{j,1}}^2\overline{{\overline\kappa_j}\psi_{j,1}}\ell^2}{T_{\max}}
+2\gamma\kappa\overline{\varphi_{j,1}}(\overline{\psi_{j,2}}-\overline{\psi_{j,1}})\ell^2,\\
g_{21}^{(j)}=&-\frac{2r_j\kappa_j\overline{\psi_{j,1}}\ell^2}{T_{\max}}\Big [(\overline{\varphi_{j,1}}(2i\omega_j-\boldsymbol{L}_{i,r_j})^{-1}\boldsymbol{{\mathcal K}_{j,1}})_1
-2\varphi_{j,1}(\boldsymbol{L}_{i,r_j}^{-1}\boldsymbol{{\mathcal K}_{j,2}})_1\Big ]\\
&+\gamma\kappa_j(\overline{\psi_{j,2}}-\overline{\psi_{j,1}})\ell^2
\Big [\overline{\varphi_{j,1}}((2i\omega_j-\boldsymbol{L}_{i,r_j})^{-1}
\boldsymbol{{\mathcal K}_{j,1}})_3+((2i\omega_j-\boldsymbol{L}_{i,r_j})^{-1}\boldsymbol{{\mathcal K}_{j,1}})_1\\
&\qquad\qquad\qquad\qquad\;-2\varphi_{j,1}(\boldsymbol{L}_{i,r_j}^{-1}\boldsymbol{{\mathcal K}_{j,2}})_3-2(\boldsymbol{L}_{i,r_j}^{-1}\boldsymbol{{\mathcal K}_{j,2}})_1\Big ],
\end{align*}
where $(\cdot)_k$ denotes the $k$-th component of the vector in brackets.

Since an explicit computation of these coefficients for any value of $T_{\max}$ is uneasy, and we are interested in large (enough)
values of $T_{\max}$, as in \cite[Sec. 4.3]{FB10} we determine the sign of ${\rm Re}\,c_1(r_j)$ via an asymptotic analysis as $T_{\max}\to +\infty$. We get
\begin{align}
{\rm Re}\,c_1(r_1) = \frac{\gamma \mv^2}{{\mathcal D}(N)}{\mathcal H}(N)+ o(1),
\qquad\;\,{\rm Re}\,c_1(r_2)=-\frac{50(\mi+\mv)^3}{\mi^2N^2T_{\max}^2} + o(T_{\max}^{-2}),
\label{AB}
\end{align}
where ${\mathcal D}(N)$ and ${\mathcal H}(N)$ are respectively given by
\begin{align}
{\mathcal D}(N)=&2\alpha(\mi\mv+\mv^2+\alpha\gamma N)(\mi^2\mv^2+2\mi\mv^3+6\alpha\gamma\mi\mv N+\mv^4+6\alpha\gamma\mv^2 N+\alpha^2\gamma^2N^2)\notag\\
&\qquad\times(\alpha^2\gamma^2N^2+3\alpha\gamma\mi\mv N+3\alpha\gamma \mv^2N+\mi^2\mv^2+2\mi\mv^3+\mv^4)\notag\\
&\qquad\times(\mi^2\mv+2\mi\mv^2+\mi\mv+\mv^3+\mv^2+\alpha\gamma\mv N)N,
\label{funct-D}
\\[2mm]
{\mathcal H}(N) = &3\alpha^5\gamma^5(\mi+\mv)^2N^5-\alpha^4\gamma^4\mv(\mi+\mv)(12\mi^2+35\mi\mv+12\mv^2)N^4\notag\\
&-\alpha^3\gamma^3\mv^3(26\mi^4+151\mi^3\mv+247\mi^2\mv^2+151\mi\mv^3+26\mv^4)N^3\notag\\
&-\alpha^2\gamma^2\mv^3(\mi+\mv)(12\mi^4+85\mi^3\mv+134\mi^2\mv^2+85\mi\mv^3+12\mv^4)N^2\notag\\
&-\alpha\gamma\mv^4(\mi+\mv)^2(\mi^4+13\mi^3\mv+35\mi^2\mv^2+13\mi\mv^3+\mv^4)N-4\mi^2\mv^7(\mi+\mv)^3.
\label{funct-N}
\end{align}
Whereas ${\mathcal D}(N)$ is always positive, the sign of ${\mathcal H}(N)$ depends on $N$.
Since ${\mathcal H}(0)<0$ and $\lim_{N \to +\infty}{\mathcal H}(N)=+\infty$, the function ${\mathcal H}$ has at least a positive zero.
We define by $N_*$ the (first) positive zero of ${\mathcal H}$. Therefore, ${\mathcal H}(N)<0$ for
$0 \leq N <N_*$. We thus conclude that, for $T_{\max}$ large enough (let us say $T_{\max}>T_{\max}^{(3)}>T_{\max}^{(2)}$, which depends on $\alpha$,
$\gamma$, $\mi$, $\mt$ and $\mv$), ${\rm Re}\,c_1(r_1)<0$ for any $0<N<N_*$, whereas
${\rm Re}\,c_1(r_2)<0$ for any $N>0$.
This completes the proof.
\end{proof}



\section{\bf Numerical results}\label{numerics}
In order to show the stability of the infected steady state numerically, we
can fix $N=300$, start from the value $r_{crit,0}=0.05625$, increase
the logistical parameter $r$ monotonically until a critical
condition is reached such that any further change would result in
instability, other parameters can be found in Table \ref{table1}. We
present the graphs of numerical solution of the system
\eqref{pde1}-\eqref{pde3} and the trajectory of the solution in the
three-dimensional $T$-$V$-$I$ space. Initial
data are $T_0=T_u+\varepsilon(\sin x\cos y),~I_0=0.0,~V_0=0.0185$. Some figures assure that
this solution approaches the limit cycle in the instability
subdomain $\mathcal P$.
In Figure \ref{N300r200} corresponding to the subdomain $\mathcal P$,
the solution approaches the periodic orbit.

\begin{figure}[ht]\hspace{0cm}
\subfloat[Densities of virus $V$]{
\begin{minipage}[]{0.5\textwidth}
\centering
\includegraphics[height=4.5cm,width=5.5cm,angle=0]{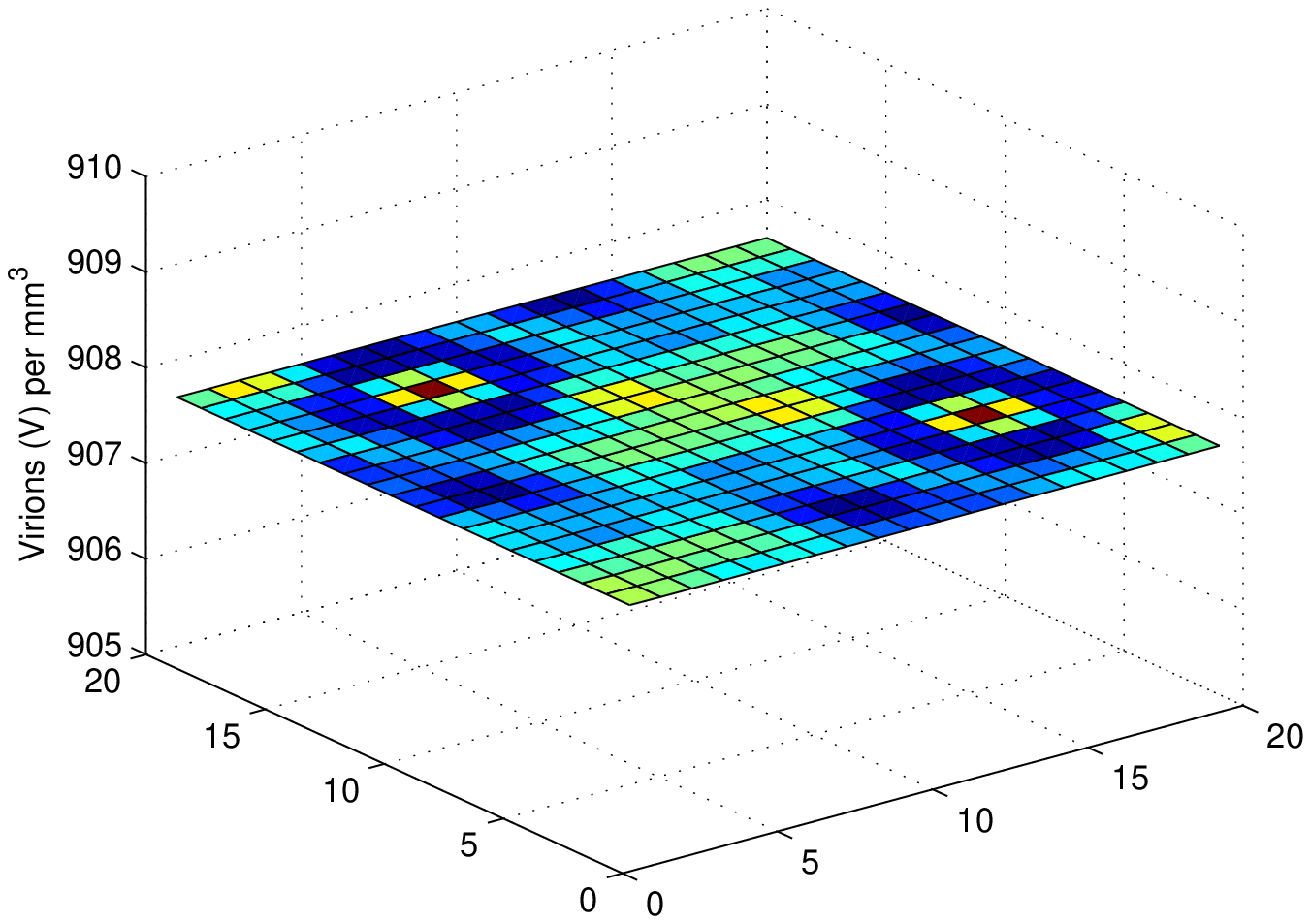}
\end{minipage}} \hspace{-1cm}
\subfloat[Densities of target cells $T$]{
\begin{minipage}[]{0.5\textwidth}
\centering
\includegraphics[height=4.5cm,width=5.5cm,angle=0]{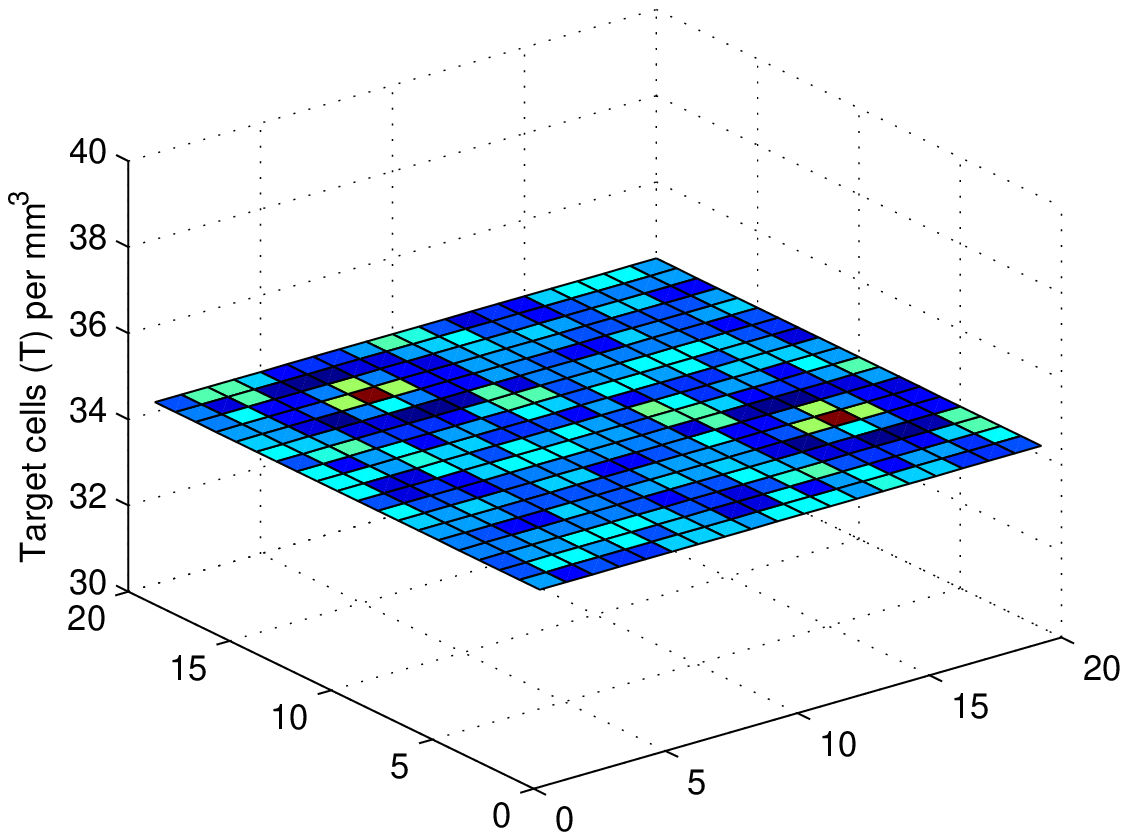}
\end{minipage}}\\
\subfloat[Free virus at (10,10)]{
\begin{minipage}[]{0.5\textwidth}
\centering
\includegraphics[height=4.5cm,width=5cm,angle=0]{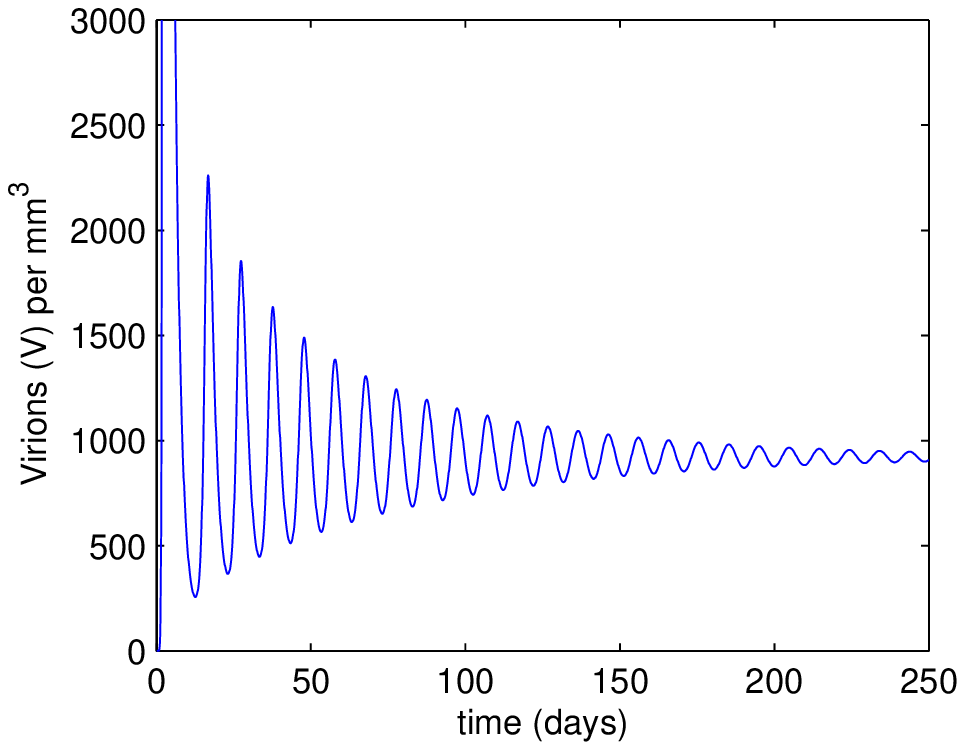}
\end{minipage}} \hspace{-1cm}
\subfloat[Target cells at (10,10)]{
\begin{minipage}[]{0.5\textwidth}
\centering
\includegraphics[height=4.5cm,width=5cm,angle=0]{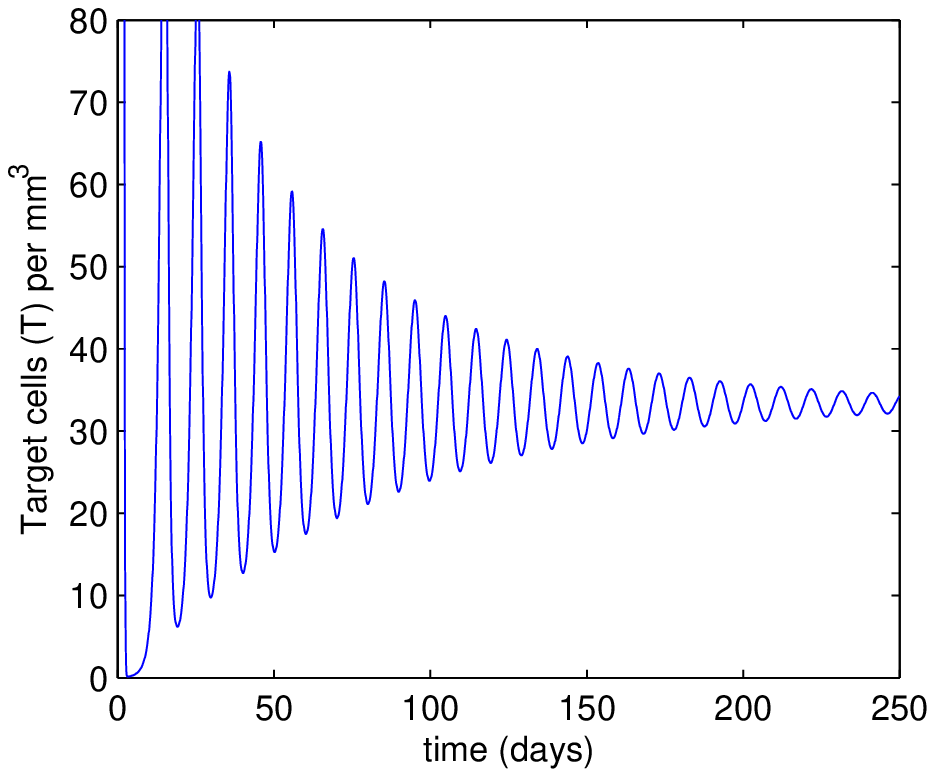}
\end{minipage}}\\\hspace{-1cm}
\subfloat[Free virus at (10,10)]{
\begin{minipage}[]{0.5\textwidth}
\centering
\includegraphics[height=4.5cm,width=6.5cm,angle=0]{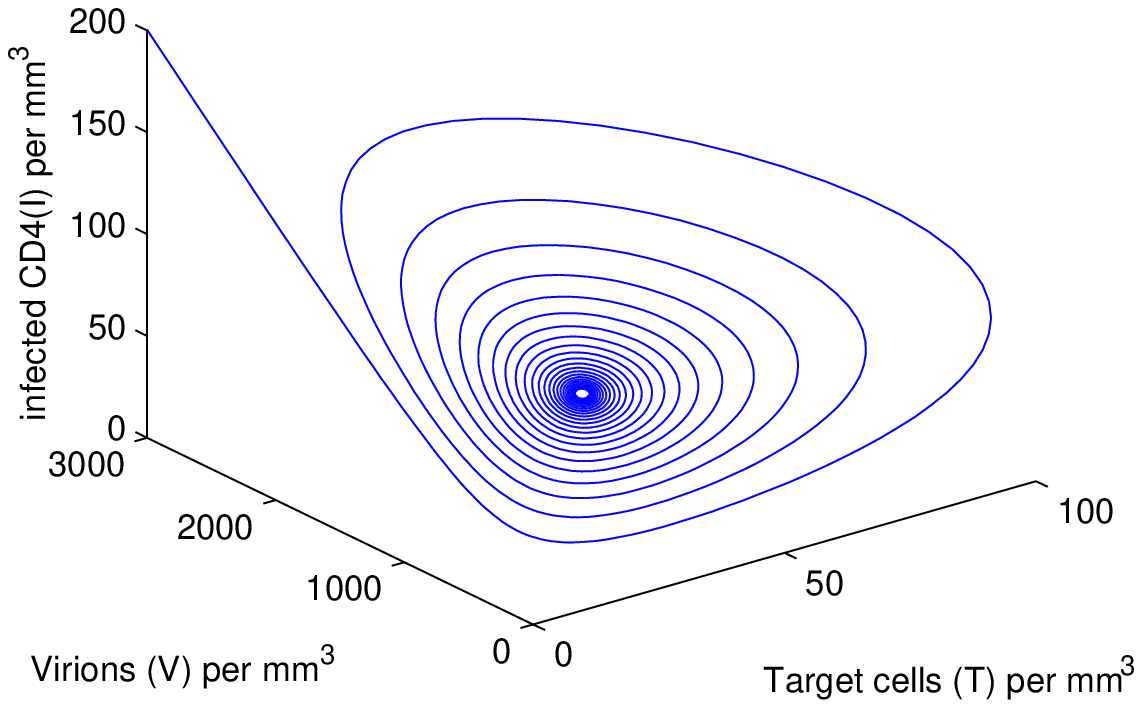}
\end{minipage}} \hspace{-1cm}
\caption{Dynamical solution of System \eqref{pde1}-\eqref{pde3} on a 20$\times$20 grid. Parameter values are $N=300$ and
$r_{\rm crit}=0.05625 <r=1.0 <r_1 = 2.1846$. The infected equilibrium is
stable.} \label{N300r1}
\end{figure}

\begin{figure}[ht]\hspace{0cm}
\subfloat[Densities of virus $V$]{
\begin{minipage}[]{0.5\textwidth}
\centering
\includegraphics[height=5cm,width=5.5cm,angle=0]{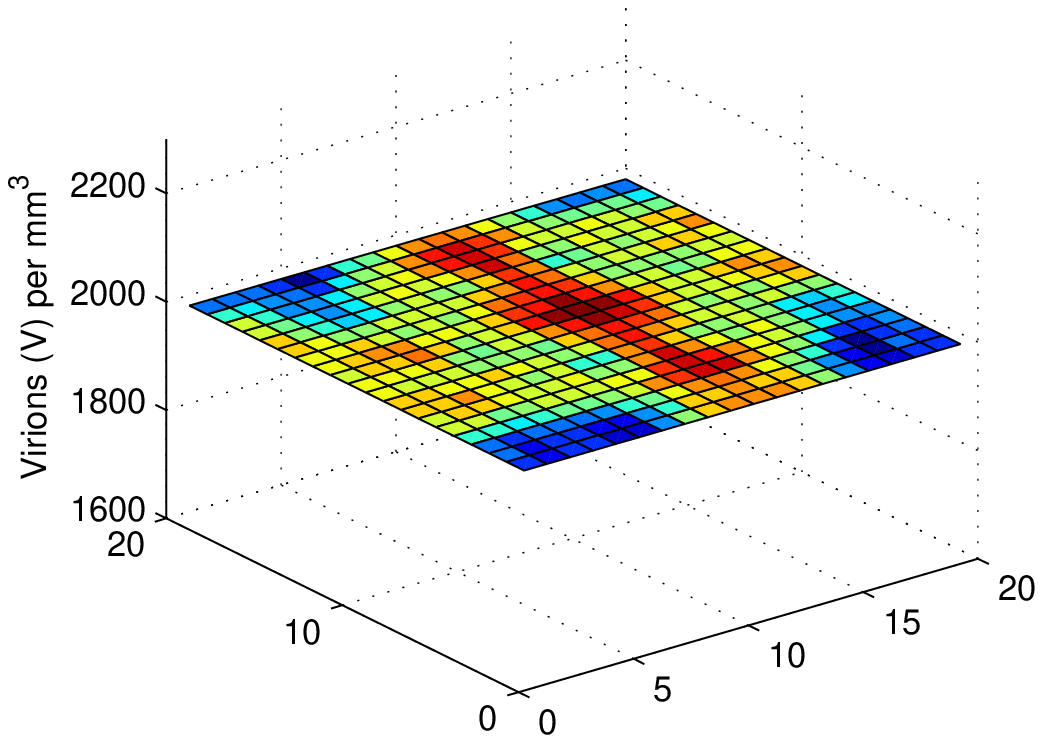}
\end{minipage}} \hspace{-1cm}
\subfloat[Densities of target cells $T$]{
\begin{minipage}[]{0.5\textwidth}
\centering
\includegraphics[height=5cm,width=5.5cm,angle=0]{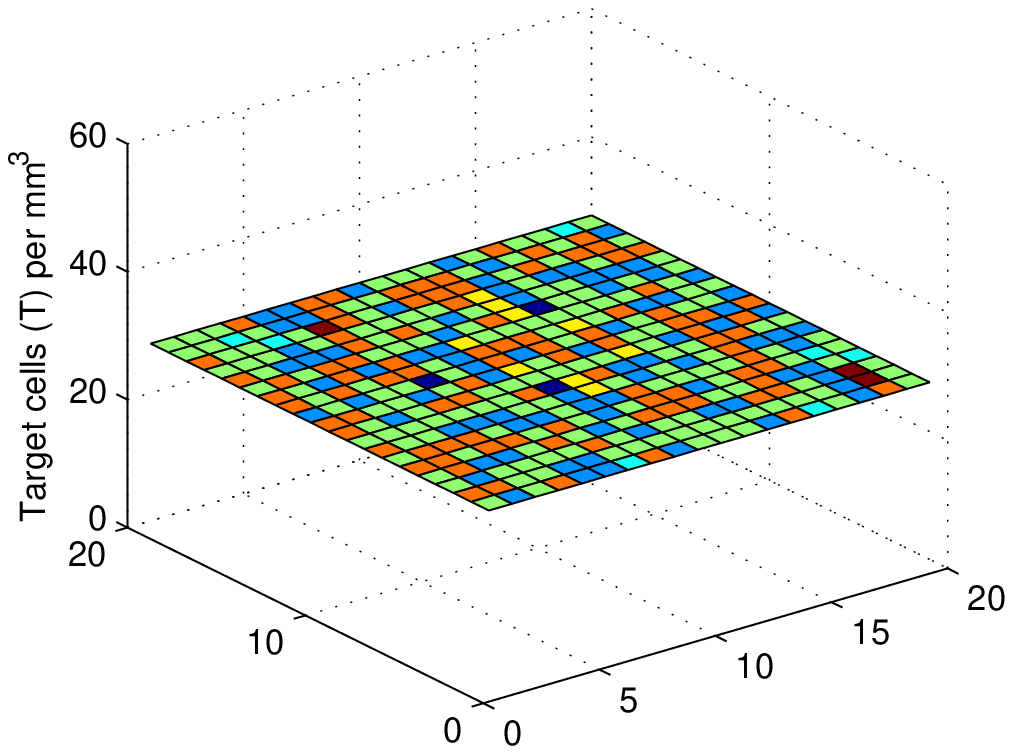}
\end{minipage}}\\
\subfloat[Free virus at (10,10)]{
\begin{minipage}[]{0.5\textwidth}
\centering
\includegraphics[height=4.5cm,width=5cm,angle=0]{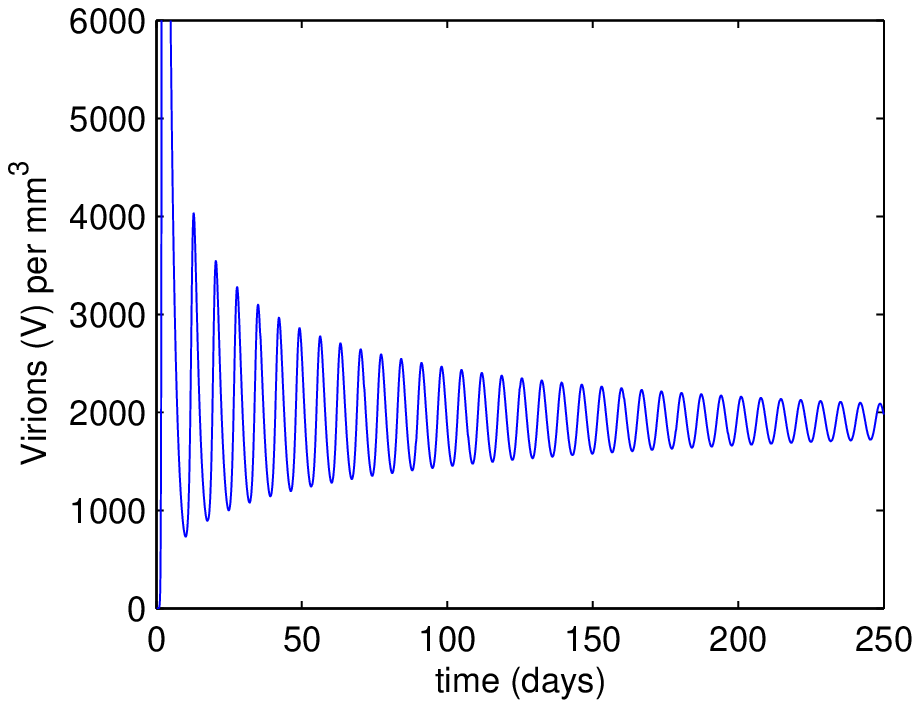}
\end{minipage}} \hspace{-1cm}
\subfloat[Target cells at (10,10)]{
\begin{minipage}[]{0.5\textwidth}
\centering
\includegraphics[height=4.5cm,width=5cm,angle=0]{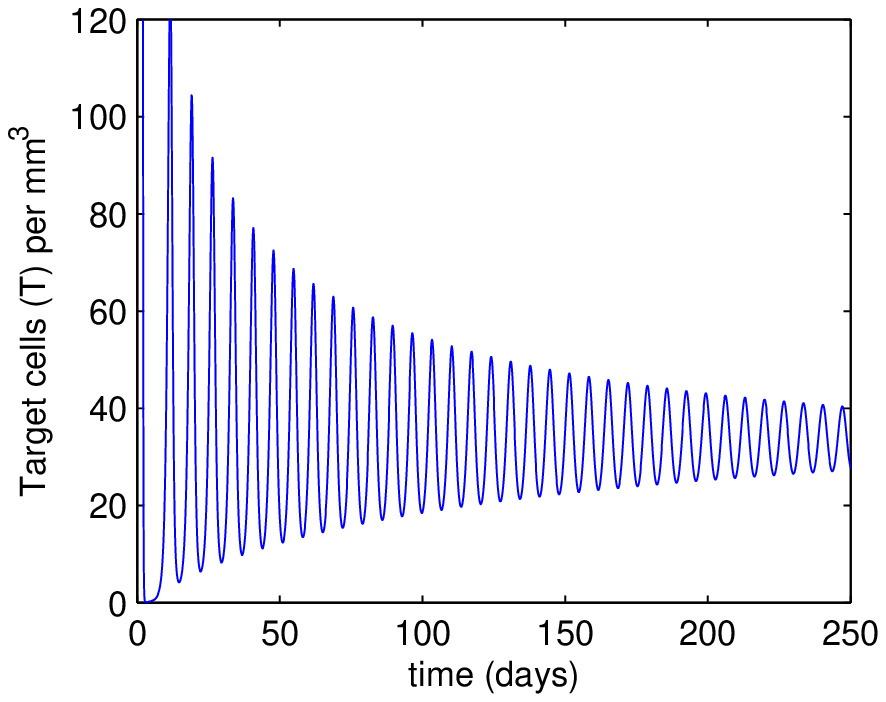}
\end{minipage}}\\\hspace{-1cm}
\subfloat[Free virus at (10,10)]{
\begin{minipage}[]{0.5\textwidth}
\centering
\includegraphics[height=5.5cm,width=6.5cm,angle=0]{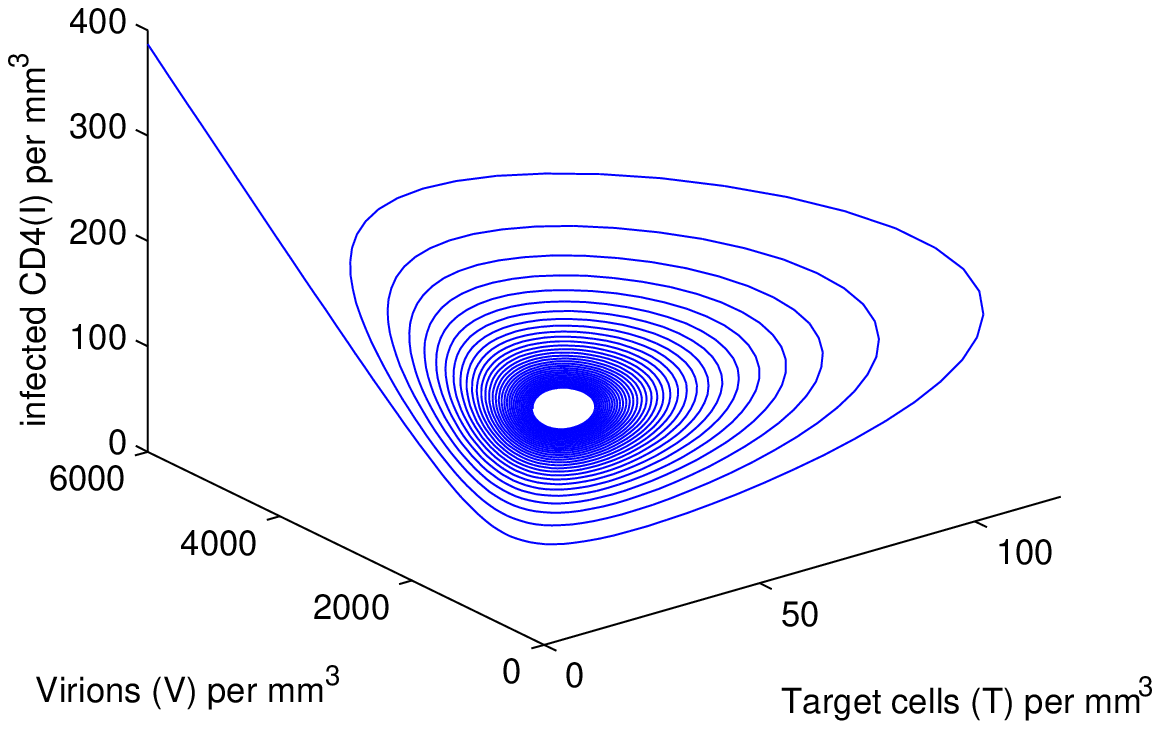}
\end{minipage}} \hspace{-1cm}
\caption{Dynamical solution of System \eqref{pde1}-\eqref{pde3} on a 20$\times$20 grid. Parameter values are $N=300$ and
$r=2.0 <r_1 = 2.1846$. The infected equilibrium is
stable.} \label{N300r2}
\end{figure}

\begin{figure}[ht]\hspace{0cm}
\subfloat[Densities of virus $V$]{
\begin{minipage}[]{0.5\textwidth}
\centering
\includegraphics[height=5cm,width=5.5cm,angle=0]{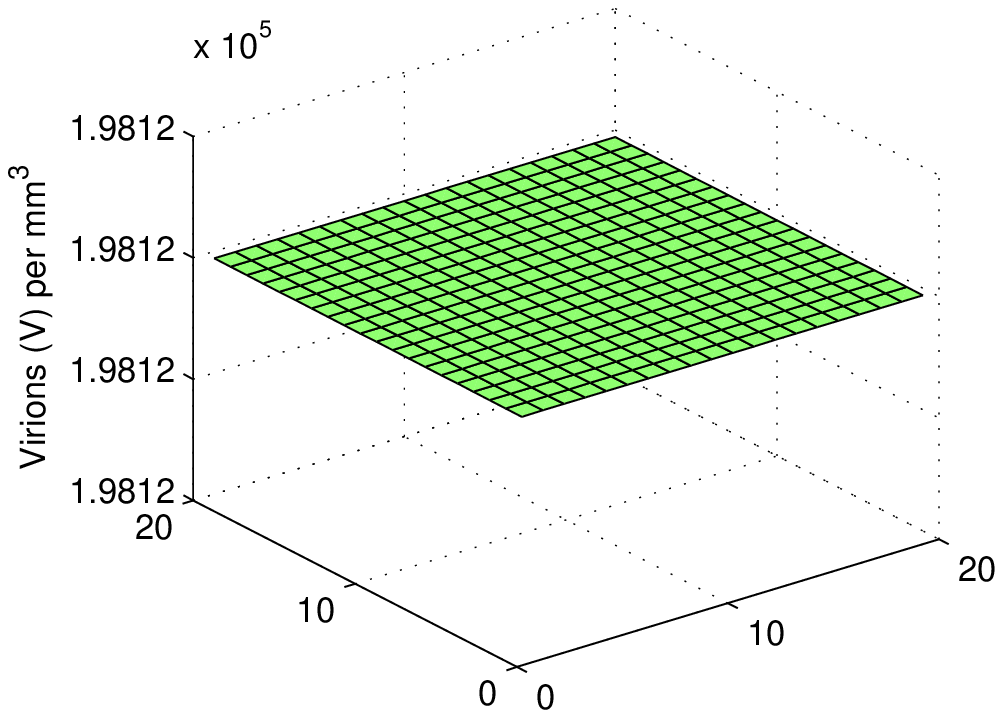}
\end{minipage}} \hspace{-1cm}
\subfloat[Densities of target cells $T$]{
\begin{minipage}[]{0.5\textwidth}
\centering
\includegraphics[height=5cm,width=5.5cm,angle=0]{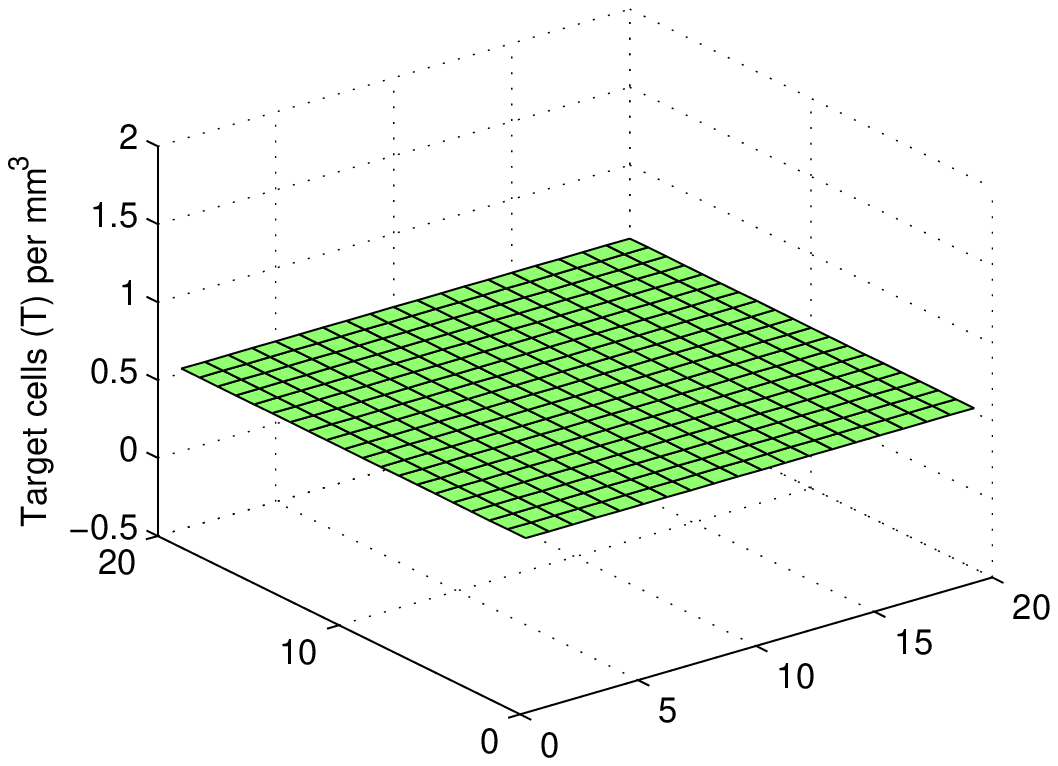}
\end{minipage}}\\
\subfloat[Free virus at (10,10)]{
\begin{minipage}[]{0.5\textwidth}
\centering
\includegraphics[height=4.5cm,width=5cm,angle=0]{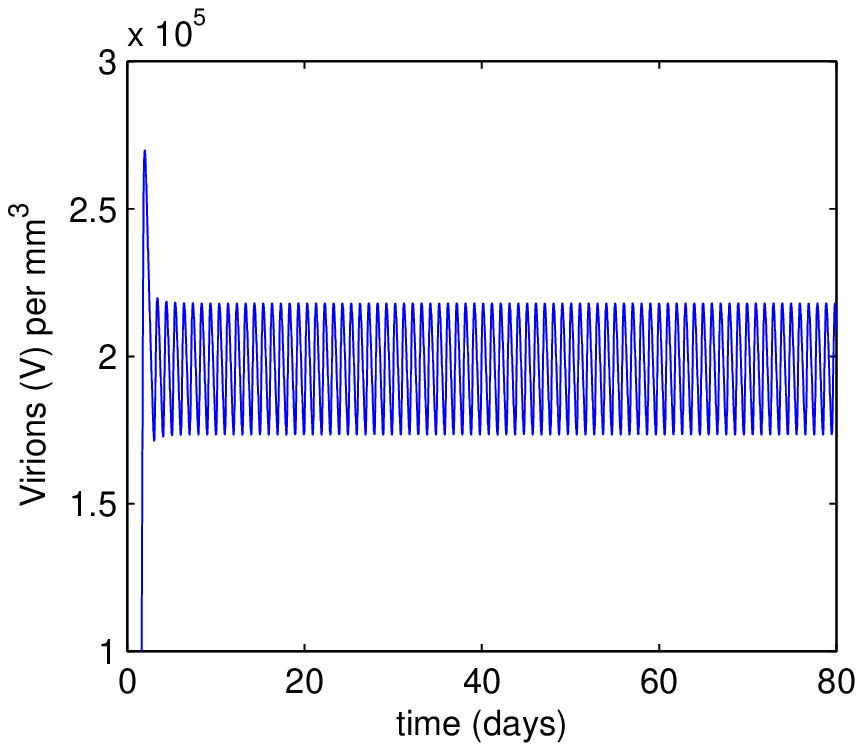}
\end{minipage}} \hspace{-1cm}
\subfloat[Target cells at (10,10)]{
\begin{minipage}[]{0.5\textwidth}
\centering
\includegraphics[height=4.5cm,width=5cm,angle=0]{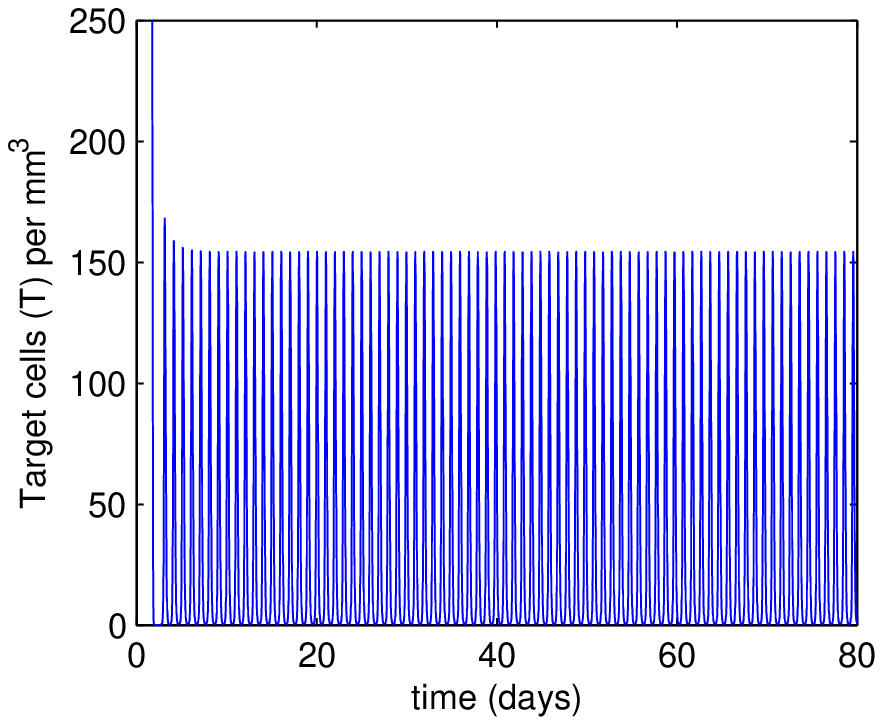}
\end{minipage}}\\ \hspace{-1cm}
\subfloat[Free virus at (10,10)]{
\begin{minipage}[]{0.5\textwidth}
\centering
\includegraphics[height=5.5cm,width=6.5cm,angle=0]{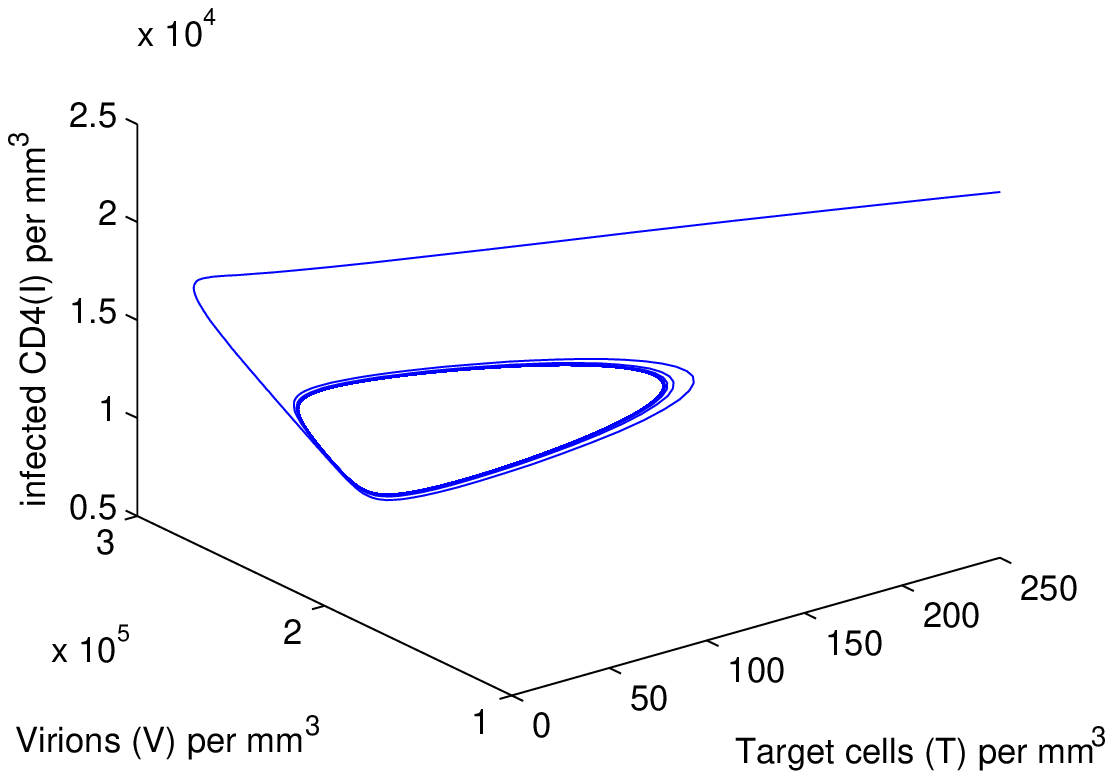}
\end{minipage}} \hspace{-1cm}
\caption{Dynamical solution of System \eqref{pde1}-\eqref{pde3} on a 20$\times$20 grid. Parameter values are $N=300$ and
$r=200.0< r_2 = 464.1225$. The infected equilibrium is
unstable.} \label{N300r200}
\end{figure}

\begin{figure}[ht]\hspace{0cm}
\subfloat[Densities of virus $V$]{
\begin{minipage}[]{0.5\textwidth}
\centering
\includegraphics[height=5cm,width=5.5cm,angle=0]{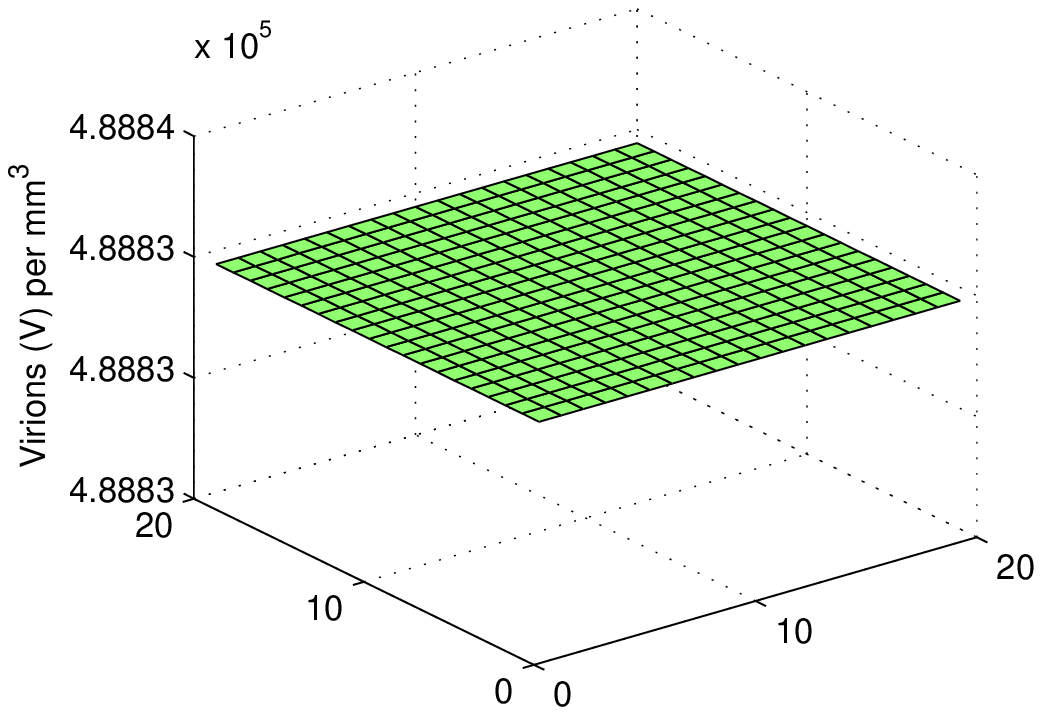}
\end{minipage}} \hspace{-1cm}
\subfloat[Densities of target cells $T$]{
\begin{minipage}[]{0.5\textwidth}
\centering
\includegraphics[height=5cm,width=5.5cm,angle=0]{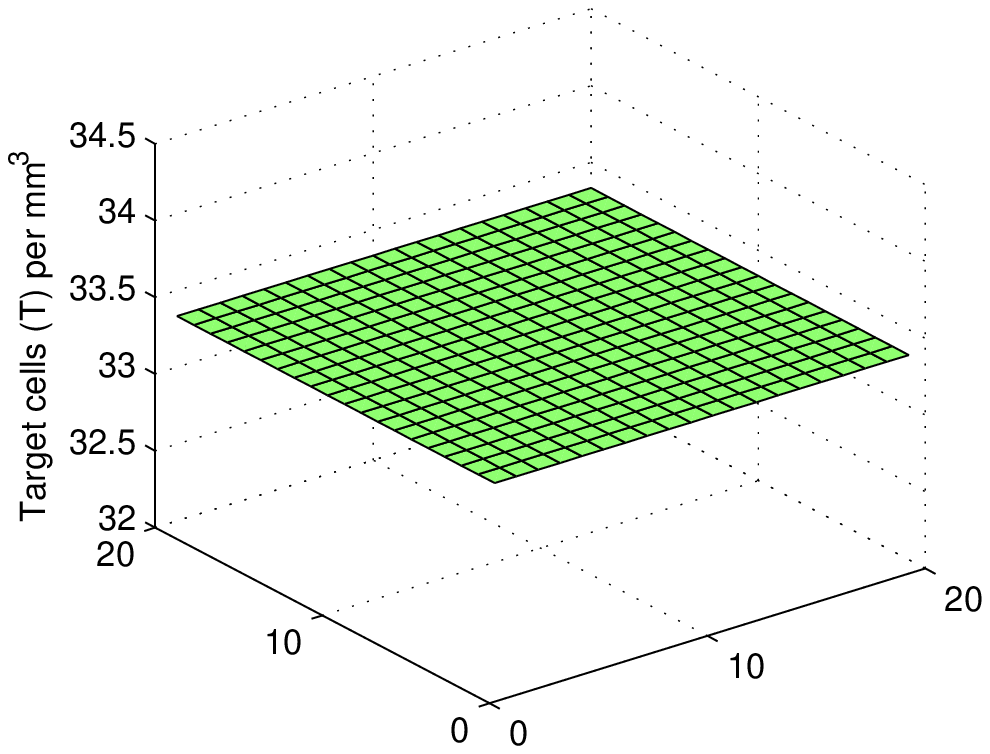}
\end{minipage}}\\
\subfloat[Free virus at (10,10)]{
\begin{minipage}[]{0.5\textwidth}
\centering
\includegraphics[height=4.5cm,width=5cm,angle=0]{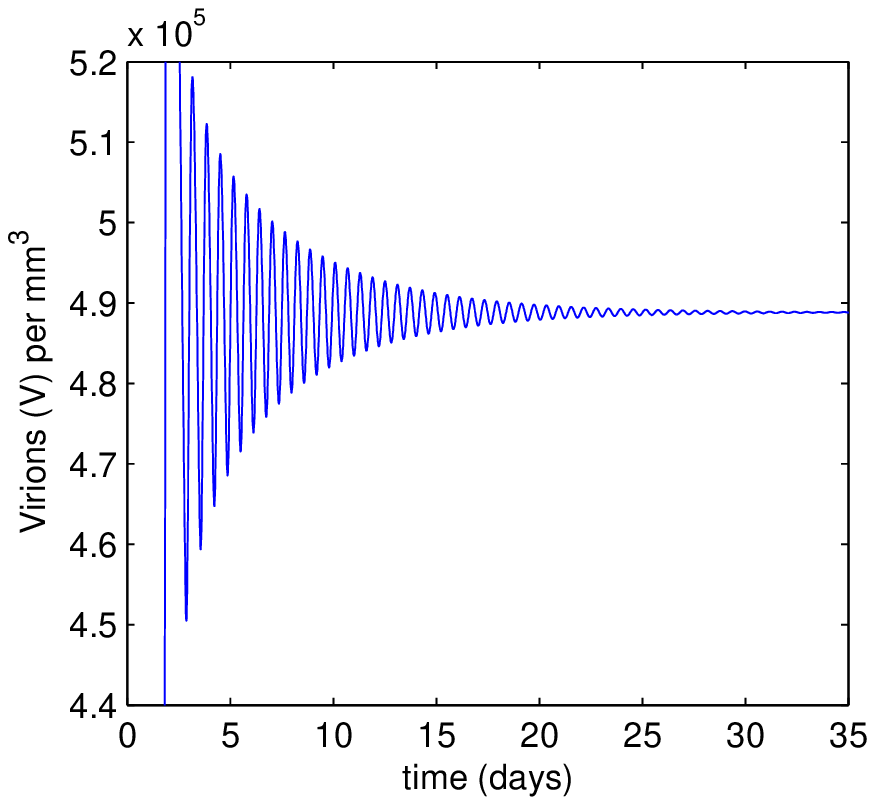}
\end{minipage}} \hspace{-1cm}
\subfloat[Target cells at (10,10)]{
\begin{minipage}[]{0.5\textwidth}
\centering
\includegraphics[height=4.5cm,width=5cm,angle=0]{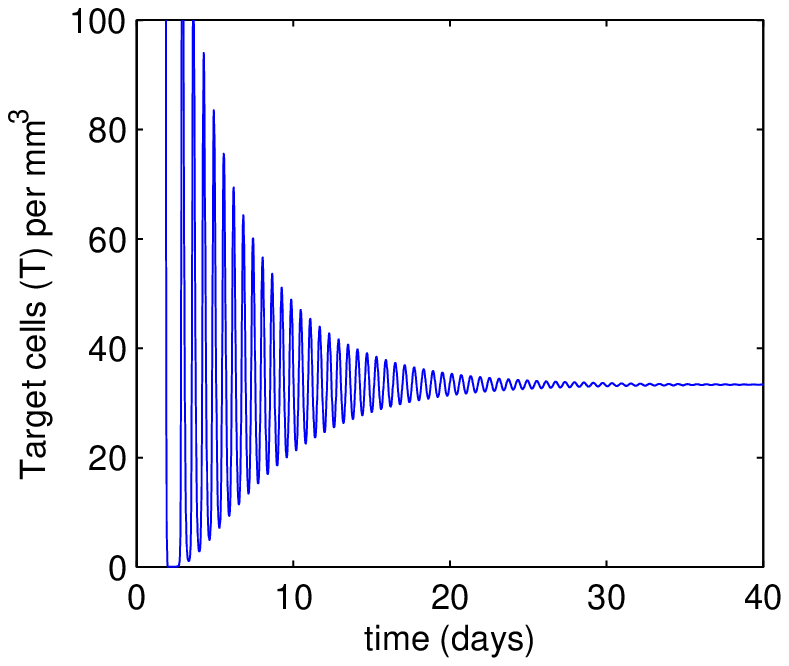}
\end{minipage}}\\\hspace{-1cm}
\subfloat[Free virus at (10,10)]{
\begin{minipage}[]{0.5\textwidth}
\centering
\includegraphics[height=5.5cm,width=6.5cm,angle=0]{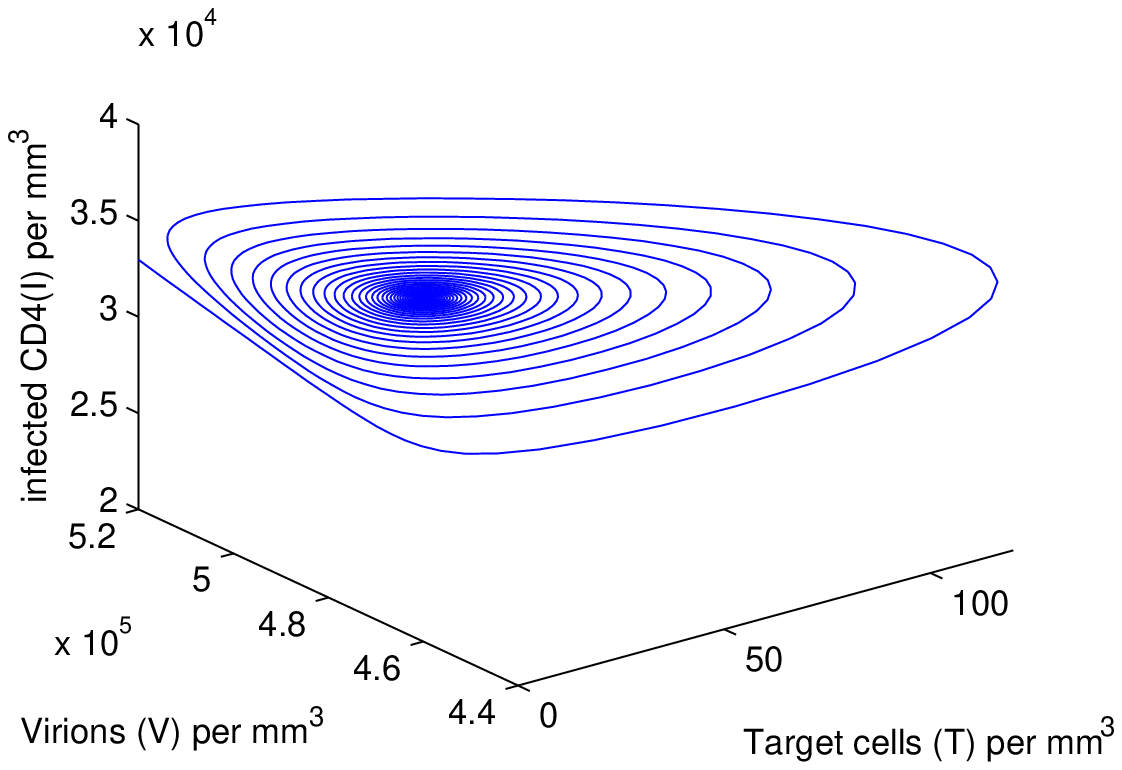}
\end{minipage}} \hspace{-1cm}
\caption{Dynamical solution of System \eqref{pde1}-\eqref{pde3} on a 20$\times$20 grid. Parameter values are $N=300$ and
$r=500.0> r_2 = 464.1225$. The infected equilibrium is
stable.} \label{N300r500}
\end{figure}
\clearpage

\appendix
\section{Eigenvalues of the Laplace operator with periodic boundary conditions}\label{Laplace eigenvalues}
\setcounter{equation}{0}

Let $A$ be the realization of the Laplace operator $\Delta$ in $L^2_{\mathbb C}$, with
$H^2_{\sharp,\mathbb C}$ as a domain. The following is a well-known result.
Nevertheless, for the reader's convenience we provide a short proof.

\begin{theorem}\label{laplace-eigenvalues}
A is a sectorial operator and its spectrum $\sigma(A)$ is a countable set of semisimple
eigenvalues. More precisely,
\begin{equation} \label{sigmaA}
\sigma(A)=\left\{-\frac{4 \pi^2}{\ell^2}(k_1^2+k_2^2):k_1,k_2\in \mathbb{N}\right\}.
\end{equation}
\end{theorem}
\begin{proof}
Fix $\lambda\in\mathbb C$, $f\in L^2_{\mathbb C}$ and consider the resolvent equation
\begin{align}\label{resolventlaplace}
\lambda u- A u=f.
\end{align}
Denote by $e_h$ the function defined by
$$
e_h(t)=\exp\left (\frac{2h\pi t}{\ell}i\right ),\qquad\;\,t\in\R,
$$
for any $h\in\mathbb Z$. Then, the functions $(x,y)\mapsto e_{h}(x)e_k(y)$ are an orthogonal
basis of $L^2_{\mathbb C}$.
Hence, any function $g\in L^2_{\mathbb C}$ can be expanded into a Fourier series as follows:
\begin{align*}
g(x,y)=\sum_{k_1,k_2\in\mathbb{N}}\left (\frac{1}{\ell^2}
\int_{\Omega_{\ell}}u(x,y)e_{-k_1}(x)e_{-k_2}(y)dx dy\right ) e_{k_1}(x)e_{k_2}(y)=:
\sum_{k_1,k_2\in\mathbb{N}}u_{k_1,k_2}e_{k_1}(x)e_{k_2}(y),
\end{align*}
for almost every $(x,y)\in\Omega_{\ell}$. Multiplying both sides of \eqref{resolventlaplace} by $e_{k_1}(x)e_{k_2}(y)$ and integrating over $\Omega_{\ell}$,
it thus follows that, if $u\in H^2_{\sharp,\mathbb C}$, then the Fourier coefficients of $u$
solves the infinitely many equations
\begin{align*}
\lambda \ell^2 u_{k_1,k_2}+4\pi^2(k_1^2+k_2^2)u_{k_1,k_2}=\ell^2 f_{k_1,k_2}.
\end{align*}
These equations are uniquely solvable if and only if $\lambda\neq -\frac{4\pi^2}{\ell^2}(k_1^2+k_2^2)$
and, in this case, we have
\begin{align*}
u_{k_1,k_2}=\frac{f_{k_1,k_2}}{\lambda+\frac{4\pi^2}{\ell^2}(k_1^2+k_2^2)}.
\end{align*}
A straightforward computation shows that the function
\begin{align*}
u(x,y)=\sum_{k_1,k_2\in\mathbb{N}}\frac{f_{k_1,k_2}}{\lambda+\frac{4\pi^2}{\ell^2}(k_1^2+k_2^2)}e_{k_1}(x)e_{k_2}(y),
\quad (x,y)\in\Omega_{\ell}
\end{align*}
is in $D(A)$ and, actually, solves the resolvent equation, when
$\lambda\neq -\frac{4\pi^2}{\ell^2}(k_1^2+k_2^2)$ for any $k_1,k_2\in\mathbb{N}$.
We have so proved that $\sigma(A)$ is given by \eqref{sigmaA}.

It is immediate to check that $\sigma(A)$ consists of eigenvalues only. Moreover,
if ${\rm Re}\,\lambda > 0$, we can estimate
\begin{align*}
\|R(\lambda,A)f\|^2_{L^2(\Omega_{\ell})}&=\sum_{k_1,k_2\in\mathbb{N}}
\frac{|f_{k_1,k_2}|^2}{\left|\lambda+\frac{4\pi^2}{\ell^2}(k_1^2+k_2^2) \right|^2}\leq
 \frac{1}{|\lambda|^2}\sum_{k_1,k_2\in\mathbb{N}}|f_{k_1,k_2}|^2
=\frac{1}{|\lambda|^2}\|f\|^2_{L^2(\Omega_{\ell})}.
\end{align*}
Proposition 2.3.1 in \cite{lunardi1995analytic} implies that $A$ is sectorial in $L^2(\Omega_{\ell})$.

Finally, we show that all the eigenvalues of $A$ are semisimple. For this purpose, let us fix one
of such eigenvalues $\lambda_0$ and let
$H=\{(k_1,k_2)\in \mathbb{N}^2: \lambda_0=-\frac{4\pi^2}{\ell^2}(k_1^2+k_2^2)\}$. Then,
\begin{align*}
(\lambda-\lambda_0)R(\lambda,A)f(x,y)&=\sum_{k_1,k_2\in H}f_{k_1,k_2}e_{k_1}(x)e_{k_2}(y)+\sum_{k_1,k_2\notin H}\frac{\lambda-\lambda_0}{\lambda+\frac{4\pi^2}{\ell^2}(k_1^2+k_2^2)}f_{k_1,k_2}e_{k_1}(x)e_{k_2}(y)\\
&=:Pf(x,y)+(\lambda-\lambda_0)D_{\lambda}f(x,y).
\end{align*}
Clearly, $P$ is the spectral projection on the eigenspace corresponding to
the eigenvalue $\lambda_0$. On the other hand,
$D_{\lambda}$ is a bounded operator in $L^2_{\mathbb C}$ uniformly with
respect to $\lambda\in B(\lambda_0, 2\pi^2/\ell^2)$. Indeed, if $(k_1,k_2)\notin H$ and $\lambda$
is as above, then
\begin{align*}
\left|\frac{4\pi^2}{\ell^2}(k_1^2+k_2^2)+\lambda\right|\geq
\left|\frac{4\pi^2}{\ell^2}(k_1^2+k_2^2)+\lambda_0\right|-\left|\lambda-\lambda_0\right|
\ge \frac{4\pi^2}{\ell^2}-|\lambda-\lambda_0|
\geq \frac{2\pi^2}{\ell^2}.
\end{align*}
Thus,
\begin{align*}
\|D_{\lambda}f\|_{L^2(\Om_{\ell})}^2
\leq \frac{\ell^4}{4\pi^4}\sum_{(k_1,k_2)\notin H}|f_{k_1,k_2}|^2\leq \frac{\ell^4}{4\pi^4}\|f\|^2_{L^2(\Omega_{\ell})},
\end{align*}
i.e., $D_{\lambda}$ is bounded, uniformly with respect to $\lambda\in B(\lambda_0, 2\pi^2/\ell^2)$.
These results imply that $\lambda_0$ is a semisimple eigenvalue of $A$. Note that the eigenspace
corresponding to $\lambda_0$ is one-dimensional if and only if $H$ is a singleton. In this case,
$\lambda_0$ is a simple eigenvalue of $A$. More precisely, the geometric multiplicity of the eigenvalue
$\lambda=\frac{4\pi^2}{\ell^2}(k_1^2+k_2^2)$ is given by $m_{\lambda}=
4\prod_{i=1}^m(r_i+1)$, where the coefficients $r_i$ are given by the following decomposition of $k_1^2+k_2^2$ in primes
\begin{equation*}
k_1^2+k_2^2=2^\alpha \prod_{i=1}^m p_i^{r_i} \prod_{j=1}^n q_j^{s_j},
\end{equation*}
with $p_i$ being primes of the form $4t+1$, and $q_j$ being primes of the form $4t+3$ (see \cite{hardy2005introduction}).
\end{proof}

The following classical result on Sturm-Liouville problems is the key tool to prove Theorem \ref{thm-2.1}(iv).
\begin{corollary}
\label{cor-app}
Let $d$ and $\mu$ be,
respectively, a positive constant and a bounded measurable function.
Further, let $B:H^2_{\sharp}\to L^2$ be the operator
defined by $Bu=d\Delta u-\mu u$ for any $u\in
H^2_{\sharp}$. Then, the spectrum of $B$ consists of
eigenvalues only. Moreover, its maximum eigenvalue $\lambda_{\max}$
is given by the following formula:
\begin{equation}
\label{variat} \lambda_{\max} = -\inf_{\psi \in H^1_{\sharp}, \psi
\not\equiv 0} \left\{\frac{d\int_{\Omega_{\ell}}|\nabla\psi|^2 dx +
\int_{\Omega_{\ell}}\mu\psi^2 dx} {\int_{\Omega_{\ell}}\psi^2 dx}
\right\}.
\end{equation}
Finally, the eigenspace corresponding to the eigenvalue
$\lambda_{\max}$ is one dimensional and contains functions which do
not change sign in $\overline{\Omega_{\ell}}$.
\end{corollary}


\end{document}